\documentclass[a4paper,11pt]{amsart}
\usepackage[utf8]{inputenc}
\usepackage[pdftex,bookmarksnumbered,bookmarksopen]{hyperref}
\hypersetup{citecolor=blue,linkcolor=blue}

\usepackage{amsfonts} 
\usepackage{indentfirst,url}
\usepackage{color}
\usepackage{graphicx}
\usepackage{amsmath,amssymb,amsthm,amscd}
\usepackage[titletoc,title]{appendix}
\usepackage{mathrsfs}
\usepackage{mathtools}
\usepackage{tikz-cd}
\usepackage{enumerate}

\author{Raffaele Carbone}
\date{}
\title{The direct image of generalized divisors and the Norm map between compactified Jacobians}
\address{Raffaele Carbone\\ Dipartimento di Matematica e Fisica\\ Universit\`{a} Roma 3\\ Largo San Leonardo Murialdo 1\\ 00146\\ Rome\\ Italy}
\email{raffaelemarco.carbone@uniroma3.it}

\theoremstyle{plain}
\newtheorem{thm}{Theorem}[section] 

\theoremstyle{definition}
\newtheorem{defi}[thm]{Definition}
\newtheorem{deflemma}[thm]{Definition/Lemma}
\newtheorem{rmk}[thm]{Remark}

\newtheorem{ex}[thm]{Example}

\theoremstyle{plain}
\newtheorem{prop}[thm]{Proposition}
\newtheorem{lemma}[thm]{Lemma}
\newtheorem{cor}[thm]{Corollary}

\newcommand{\Spec}{\operatorname{Spec}}
\newcommand{\Supp}{\operatorname{Supp}}
\newcommand{\Pic}{\operatorname{Pic}}

\newcommand{\WDiv}{\operatorname{WDiv}}
\newcommand{\GPic}{\operatorname{GPic}}
\newcommand{\GDiv}{\operatorname{GDiv}}
\newcommand{\Prin}{\operatorname{Prin}}
\newcommand{\CDiv}{\operatorname{CDiv}}

\newcommand{\Nm}{\operatorname{Nm}}

\newcommand{\Hilb}{\operatorname{Hilb}}

\newcommand{\lHilb}{\operatorname{{}^\ell Hilb}}
\newcommand{\lrho}{\operatorname{{}^\ell \rho}}

\newcommand{\Jgen}{\operatorname{\textit{G}\mathbb{J}}}
\newcommand{\JgenSch}{\operatorname{\textit{GJ}}}
\newcommand{\Jbar}{\overline{\mathbb{J}}}
\newcommand{\JbarSch}{\overline{J}}
\newcommand{\Jtf}{\overline{\mathbb{J}}_{tf}}
\newcommand{\JtfSch}{\overline{J}_{tf}}

\newcommand{\rk}{\operatorname{rk}}

\newcommand{\rr}{\operatorname{\textbf{rank}}}
\newcommand{\mult}{\operatorname{mult}}

\newcommand{\Prbar}{\overline{\Pr}}

\newcommand{\Fitt}{\operatorname{Fitt}}

\newcommand{\Nn}{\mathcal{N}}
\newcommand{\Z}{\mathbb{Z}}
\newcommand{\Zz}{\mathcal{Z}}

\newcommand{\Kk}{\mathcal{K}}

\newcommand{\A}{\mathcal{A}}

\newcommand{\rvect}{{\underline{r}}}

\newcommand{\Ee}{\mathcal{E}}
\newcommand{\Oo}{\mathcal{O}}
\newcommand{\Ii}{\mathcal{I}}

\newcommand{\Jj}{\mathcal{J}}

\newcommand{\Ff}{\mathcal{F}}

\newcommand{\Mm}{\mathcal{M}}

\newcommand{\Ll}{\mathcal{L}}

\newcommand{\p}{\mathfrak{p}}
\newcommand{\q}{\mathfrak{q}}

\DeclareMathOperator{\Hhom}{\mathcal{H}\textit{om}}

\DeclareMathOperator{\id}{id}

\usetikzlibrary{decorations.markings}

\makeatletter
\tikzcdset{
  open/.code     = {\tikzcdset{hook, circled};},
  closed/.code   = {\tikzcdset{hook, slashed};},
  open'/.code    = {\tikzcdset{hook', circled};},
  closed'/.code  = {\tikzcdset{hook', slashed};},
  circled/.code  = {\tikzcdset{markwith = {\draw (0,0) circle (.375ex);}};},
  slashed/.code  = {\tikzcdset{markwith = {\draw[-] (-.4ex,-.4ex) -- (.4ex,.4ex);}};},
  markwith/.code ={
    \pgfutil@ifundefined%
    {tikz@library@decorations.markings@loaded}%
    {\pgfutil@packageerror{tikz-cd}{You need to say %
      \string\usetikzlibrary{decorations.markings} to use arrows with markings}{}}{}%
    \pgfkeysalso{/tikz/postaction = {
      /tikz/decorate,
      /tikz/decoration={markings, mark = at position 0.5 with {#1}}}
    }
  },
}
\makeatother

\makeatletter
\providecommand*{\twoheadrightarrowfill@}{%
  \arrowfill@\relbar\relbar\twoheadrightarrow
}
\providecommand*{\twoheadleftarrowfill@}{%
  \arrowfill@\twoheadleftarrow\relbar\relbar
}
\providecommand*{\xtwoheadrightarrow}[2][]{%
  \ext@arrow 0579\twoheadrightarrowfill@{#1}{#2}%
}
\providecommand*{\xtwoheadleftarrow}[2][]{%
  \ext@arrow 5097\twoheadleftarrowfill@{#1}{#2}%
}
\makeatother

\begin{document}

\begin{abstract}
Given a finite, flat morphism between embeddable noetherian schemes of pure dimension 1, we define the notion of direct and inverse image for generalized divisors and generalized line bundles. In the case when we deal with (possibly reducible, non-reduced) projective curves over a field and the codomain curve is smooth, we introduce the compactified Jacobians parametrizing torsion-free rank-1 sheaves and we study the Norm and inverse image maps between compactified Jacobians. Finally, we introduce and study the Prym stack defined as the kernel of the Norm map.
\end{abstract}

\maketitle

\setcounter{tocdepth}{1} 
\tableofcontents

\section*{Introduction}\label{SectionIntroduction}

Let $\pi: X \rightarrow Y$ be a finite, flat morphism between noetherian curves. The first aim of this paper is to define and study the direct and inverse image for generalized divisors and generalized line bundles on $X$ and $Y$. Moreover, in the case when $X$ and $Y$ are projective curves over a field $k$, our aim is to define and study the direct and inverse image for families of generalized divisors, and the notion of Norm and inverse image for torsion-free rank-1 sheaves.

Recall that the notion of Cartier divisor has been generalized by Hartshorne in his papers \cite{Har94} and \cite{Har07}, in order to include as effective divisors any codimension-one subscheme without embedded points, defined possibly by more than a single equation. A \textit{generalized divisor} on a curve $X$ of pure dimension 1 is a nondegenerate fractional ideal of $\Oo_X$-modules; generalized divisors up to linear equivalence are equivalent to \textit{generalized line bundles}, i.e. pure coherent sheaves which are locally free of rank 1 at each generic point. Generalized line bundles can be generalized further by considering \textit{torsion-free sheaves ``of rank 1''}, with different notions of rank.

The notions of direct and inverse image for Cartier divisors are well-known; these notions are compatible with linear equivalence, giving rise to analogous well-known maps for line bundles, called the Norm and the inverse image maps between the Jacobian schemes (see for example \cite[§21]{EgaIV4} and \cite{HP}). To the best of our knowledge, the same notions for generalized divisors, generalized line bundles and torsion-free rank-1 sheaves have not been discussed so far.

In the first part of the paper we collect and review in a coherent framework definitions and results about the Norm map, generalized divisors, torsion-free rank-1 sheaves; we also recall the construction and the properties of the Fitting ideal, a construction used in algebraic geometry to produce a direct image of a morphism of schemes. In the second part, basing upon geometric intuitions, we propose definitions of the direct and inverse image for generalized divisors and generalized line bundles, and we study their properties. Then, restricting to the case when $X$ and $Y$ are projective curves over a field, we consider the same notions for families of effective generalized divisors, parametrized by the Hilbert scheme. In particular, the direct image between Hilbert schemes can be defined only if the curve $Y$ is smooth. In the third part of the paper, under the condition that $Y$ is smooth, we study the Norm and the inverse image maps between the compactified Jacobians parametrizing torsion-free rank-1 sheaves on $X$ and $Y$; assuming that $X$ is reduced with locally planar singularities, we also study the fibers of the Norm map.

\vspace{1em}
The structure of this paper is the following. In Section \ref{SectionReviewNorm} we review in details the definition of direct and inverse image for Cartier divisors and the definition of Norm map for line bundles. In Section \ref{SectionReviewJacobians} we review and compare different classes of coherent sheaves with well-defined rank and degree; in particular we distinguish among line bundles, generalized line bundles, torsion-free rank-1 sheaves and torsion-free sheaves of polarized rank 1. The moduli stacks corresponding to such classes are introduced and compared. In Section \ref{SectionReviewDivisors} we review the theory of generalized divisors on a curve and its connection with the theory of generalized line bundles, via the Abel map; we also introduce the geometric objects that parametrize well-behaved families of generalized divisors on a fixed curve. In Section \ref{SectionReviewFitting} we recall briefly the notion of Fitting ideal associated to any coherent sheaf. In Section \ref{SectionDirectImageSet} we introduce and discuss in details the notion of direct and inverse image for generalized divisors and generalized line bundles. In Section \ref{SectionDirectImageFamilies} we study the notion of direct and  inverse image for families of effective generalized divisors. In Section \ref{SectionNormOnJbar} we define the Norm map and the inverse image map between compactified Jacobians, and their relation with the generalized Abel map. Finally, in Section \ref{SectionPrym} we study the fibers of the Norm map and we introduce the Prym stack as the fiber over the trivial sheaf.

\vspace{1em}

The motivation of this work comes from  the Hitchin fibration for the moduli space of $G$-Higgs bundles on a fixed smooth and projective curve $C$, for $G$ varying among classical groups such as GL, SL, PGL, SO, Sp. When $G=$ GL, the Hitchin fibres are compactified Jacobians of certain (possibly singular) covers of $C$, called spectral curves, associated to the fibre. For others $G$, a description of the generic Hitchin fibres (i.e. those for which the singular locus of the spectral curve is empty or as small as possible) is given with the help of a Prym variety, i.e. the locus in $J(X)$ of bundles whose Norm with respect to a given finite, flat morphism is zero. The study of the Norm map between compactified Jacobians allows us to extend this kind of description to the remaining Hitchin fibres. 

\vspace{1em}

\textbf{Acknowledgements}. This work is part of my PhD thesis, done under the supervision of Filippo Viviani at Roma Tre University. I am sincerely grateful to him for giving me the problem and for his continuous guidance and proofreading. I would also like to thank Edoardo Sernesi for helpful comments and discussions.

\vspace{1em}

\textbf{Notation}. In the absence of further specifications, by \textit{curve} we refer to a noetherian scheme of pure dimension 1 which is embeddable (i.e. it can be embedded as a closed subscheme of a regular scheme). This implies that the canonical (or dualizing) sheaf $\omega$ of the curve is well defined. 

\section{Review of the Norm map}\label{SectionReviewNorm}
In this section, we resume the definition and properties of the direct and inverse image for Cartier divisors and the Norm map for line bundles, associated to a finite, flat morphism between curves. For a complete treatment, the standard reference is \cite[§21]{EgaIV4} together with \cite[§6.5]{EgaII}. We start with the definition of the norm at the level of sheaves of algebras.

\vspace{1em}

Let $\pi: X \rightarrow Y$ be a finite, flat morphism between curves of degree $n$. Since $Y$ is noetherian, this is equivalent to require that $f$ is finite and locally free, i.e. that  $\pi_* \Oo_X$ is a locally free $\Oo_Y$-algebra \cite[\href{https://stacks.math.columbia.edu/tag/02K9}{Tag 02K9}]{stacks-project}.
\begin{defi}\label{NormOfSheaves}
The sheaf $\pi_* \Oo_X$ is endowed with a homomorphism of $\Oo_Y$-modules, called the norm and defined on local sections by:
\begin{align*}
\Nn_{Y/X}: \hspace{2em} \pi_* \Oo_X &\longrightarrow \Oo_Y \\
s &\longmapsto \det(\cdot s)
\end{align*}
where $\cdot s: \pi_* \Oo_X \rightarrow  \pi_* \Oo_X $  is the multiplication map induced by $s$ and $\det(\cdot s)$ is given locally by the determinant of the matrix with entries in $\Oo_Y$ associated to $\cdot s$.
\end{defi}
 By the standard properties of determinants, for local sections $s,s'$  of $\pi_*\Oo_X$ and any section $\mu$ of $\Oo_Y$, we have:
\begin{align}
\Nn_{Y/X}(s \cdot s') = \Nn_{Y/X}(s) \cdot \Nn_{Y/X}(s'), \hspace{3em} \Nn_{Y/X}(\mu s)= \mu^n\Nn_{Y/X}(s). \label{normpr}
\end{align}

\vspace{1em}

Before giving the definitions of the present section, we need a technical lemma.
\begin{lemma}\label{TrivialOnSaturated}
	Let $\pi: X \rightarrow Y$ be a finite, flat morphism of degree $n$ between curves and let $\Ll$ be an invertible $\Oo_X$-module. Then, $\pi_*\Ll$ is an invertible $\pi_*\Oo_X$-module and there exists an open affine cover  $\{ V_i \}_{i \in I}$ of $Y$ s.t. $\pi_*\Oo_X$ is trivial on each $V_i$ and $\pi_*\Ll$ is trivial both as a $\pi_*\Oo_X$-module and as an $\Oo_Y$-module on each $V_i$.
\end{lemma}
\begin{proof}
	By \cite[\href{https://stacks.math.columbia.edu/tag/02K9}{Tag 02K9}]{stacks-project},  $\pi_*\Oo_X$ is a locally free $\Oo_Y$-module; denote with $\{ W_{\alpha} \}_{\alpha \in A}$ an open affine cover such that  ${\pi_*\Oo_X}_{|W_{\alpha}} \simeq {\Oo_Y^n}_{|W_{\alpha}}$ for each $\alpha \in A$. By \cite[Proposition 6.1.12]{EgaII}, $\pi_*\Ll$ is an invertible $\pi_*\Oo_X$-module; denote with $\{ W'_\beta \}_{\beta \in B}$ an open affine cover such that $ {\pi_*\Ll}_{|W'_\beta} \simeq {\pi_*\Oo_X}_{|W'_\beta} $ for each $\beta \in B$. Let $ \{ V_{\alpha, \beta} = W_\alpha \cap W'_\beta \}_{(\alpha, \beta) \in A \times B}$ be a common refinement. Then, for $I=A \times B$, $\{ V_i \}_{i \in I}$ is an open affine cover of $Y$ such that $\pi_*\Oo_X$ is trivial on each open of the cover and $\pi_*\Ll$ is trivial both as $\pi_*\Oo_X$-module and as  $\Oo_Y$-module on each open.
\end{proof}

\subsection{Direct and inverse image of Cartier divisors}
We recall now the definitions of direct and inverse image for Cartier divisors. For any curve $X$, denote with $\Kk_X$ the sheaf of total quotient rings of the curve. Recall that the set of \textit{Cartier divisors} on $X$ is the set of global sections of the quotient sheaf of multiplicative groups $\Kk_X^*/\Oo_X^*$: \[
\CDiv(X) = \Gamma(X, \Kk_X^*/\Oo_X^*).
\]
Although the group operation on $\Kk_X^*/\Oo_X^*$ is multiplication, the group operation on $\CDiv(X)$ is denoted additively.  The group of Cartier divisors of $X$ contains the subgroup $\Prin(X)$ of \textit{principal divisors} defined as the image of the canonical homomorphism \[
\Gamma(X, \Kk_X^*) \longrightarrow \Gamma(X, \Kk_X^*/\Oo_X^*).
\]
It is well known \cite[§21.2]{EgaIV4}  that the set of Cartier divisors is in one-to-one correspondence with the set of invertible fractional ideals, i.e. the set of subsheaves $\Ii \subseteq \Kk_X$ that are also invertible $\Oo_X$-modules \cite[Proposition 21.2.6]{EgaIV4}. If $D \in \Gamma(X, \Kk_X^*/\Oo_X^*)$ is represented by an open cover $\{U_i\}_{i \in I}$ and a collection of sections $f_i \in \Gamma(U_i, \Kk_X^*)$ such that $f_i/f_j \in \Gamma(U_i \cap U_j, \Oo_X^*)$, the corresponding fractional ideal $\Ii_D$ is the sub $\Oo_X$-module of $\Kk_X$ equal to $\Oo_{X|U_i}\cdot f_i$ on any $U_i$.\footnote{Differently from Grothendieck's EGA, we pick $f_i$ instead of $f_i^{-1}$. This does not affect the other results.} Under this correspondence, the sum of Cartier divisors corresponds to the multiplication of fractional ideals.

\begin{deflemma}\label{DirectInverseImageForCartDef}
Let $\pi: X \rightarrow Y$ be a finite, flat morphism between curves and let $\pi^\sharp: \Oo_Y \rightarrow \pi_* \Oo_X$ be the associated canonical map of sheaves of modules.

Let $D \in \CDiv(X)$ be a Cartier divisor on $X$ corresponding to the invertible fractional ideal $\Ii$ and let $\{ V_i \}_{i \in I}$ be an affine cover as in Lemma \ref{TrivialOnSaturated}. Then, on each $V_i$, $\pi_*\Ii_{|V_i}$ is equal to the subsheaf $h_i \cdot (\pi_*\Oo_X)_{|V_i}$ of $(\pi_*\Kk_X)_{|V_i}$ generated by a meromorphic regular section $h_i=f_i/g_i$, with $f_i, g_i \in \Gamma(V_i, \pi_* \Oo_X^*)$. The \textit{direct image of $D$}, denoted $\pi_*(D)$, is the Cartier divisor on $Y$ corresponding to the fractional ideal generated on any $V_i$ by the meromorphic regular section $\Nn_{Y/X}(f_i)/\Nn_{Y/X}(g_i) \in \Gamma(V_i,\Kk_Y)$.

Let $M \in \CDiv(Y)$ be a Cartier divisor on $Y$ corresponding to the invertible fractional ideal $\Jj \subseteq \Kk_Y$, and let $\{ V_i \}_{i \in I}$ be an affine cover of $Y$ such that, on each $V_i$, $\Jj_{|V_i}$ is equal to the fractional ideal of $\Oo_{Y|V_i}$-modules generated by a meromorphic regular section $u_i=s_i/t_i$ with $s_i, t_i \in \Gamma(V_i, \Oo_Y^*)$. The \textit{inverse image} of $M$, denoted $\pi^*(M)$, is the Cartier divisor on $X$ corresponding to the fractional ideal generated on any $U_i = \pi^{-1}(V_i)$ by the meromorphic regular section $\pi^\sharp(s_i)/\pi^\sharp(t_i) \in \Gamma(U_i, \Kk_X)$.
\end{deflemma}
\begin{proof}
The definition of direct image is a reformulation of the one given in \cite[§21.5.5]{EgaIV4}. The fact that $\Nn_{Y/X}(f_i)/\Nn_{Y/X}(g_i)$ is a regular meromorphic section follows from the discussion in \cite[§21.5.3]{EgaIV4}; the definition is independent of the choice of the $h_i$'s since the norm of sheaves is multiplicative.

The definition of inverse image is a reformulation of \cite[Definition 21.4.2]{EgaIV4}, obtained as a consequence of \cite[§21.4.3]{EgaIV4}, together with the result of \cite[Proposition 21.4.5]{EgaIV4} that ensures that $\pi^\sharp(s_i)/\pi^\sharp(t_i)$ is regular.
\end{proof}

\begin{prop}
Let $\pi: X \rightarrow Y$ be a finite, flat morphism of curves of degree $n$. The direct and inverse image for Cartier divisors induce homorphisms of groups:
\begin{align*}
\pi_*: \CDiv(X) &\longrightarrow \CDiv(Y) \\
\pi^*: \CDiv(Y) &\longrightarrow \CDiv(X),
\end{align*}
such that $\pi_* \circ \pi^*$ is the multiplication map by $n$.
\end{prop}
\begin{proof}
The direct image is a homomorphism thanks to \cite[§21.5.5.1]{EgaIV4}, while the inverse image is a homomorphism as a consequence of \cite[Definition 21.5.5.1]{EgaIV4}. The result on $\pi_* \circ \pi^*$ is stated in \cite[Proposition 21.5.6]{EgaIV4}
\end{proof}

\subsection{Norm of line bundles}
We come now to the definition of the Norm map for invertible $\Oo_X$-modules.

\begin{deflemma}
Let $\pi: X \rightarrow Y$ be a finite, flat morphism between curves and let $\Ll$ be an invertible $\Oo_X$-module. Let $\{ V_i \}_{i \in I}$ be an affine cover of $Y$ as in Lemma \ref{TrivialOnSaturated}. In particular, there is for any $i \in I$ an isomorphism $\lambda_i: (\pi_* \Ll)_{|V_i} \rightarrow (\pi_*\Oo_X)_{|V_i}$. For any $i,j \in I$, the isomorphism $\omega_{ij} :=
\lambda_i \circ \lambda_j^{-1}$ can be 
interpreted as a section of $\pi_* \Oo_X$ over $V_i \cap V_j$. The collection of norms $\{\Nn_{Y/X}
(\omega_{ij})\}_{i,j \in I}$ is a 1-cocycle with values in $\Oo_Y^*$.

The cocyle $\{\Nn_{Y/X}
(\omega_{ij})\}_{i,j \in I}$ defines up to isomorphism an invertible $\Oo_Y$-module, which is called the \textit{Norm of }$\Ll$ \textit{relative to} $\pi$ and  is denoted as $\Nm_\pi(\Ll)$ or $\Nm_{Y/X}(\Ll)$.
\end{deflemma}
\begin{proof}
	If $\Ll'$ is an invertible $\Oo_X$-module isomorphic to $\Ll$ through an isomorphism $h: \Ll' \rightarrow \Ll$, then a local trivialization for $\pi_*\Ll'$ over $V_i$ is given by $ \lambda_i \circ (\pi_* h) $. Running over all  $i\in I$, the resulting 1-cocycle $\{\Nn_{Y/X}(\lambda_i \circ \pi_*h \circ \pi_*h^{-1}\circ  \lambda_j^{-1} ) \}_{i,j \in I}$ is the same as for $\Ll$.
\end{proof}

Recall that, for any curve $X$, the Picard group of $X$ is the set $\Pic(X)$ of isomorphism classes of invertible $\Oo_X$-modules, endowed with the operation of tensor product. 

\begin{prop}
	Let $\pi: X \rightarrow Y$ be a finite, flat morphism of curves of degree $n$. The Norm and the inverse image map for line bundles induce homomorphism of groups:
\begin{align*}
\Nm_\pi: \Pic(X) &\longrightarrow \Pic(Y) \\
\pi^*: \Pic(Y) &\longrightarrow \Pic(X),
\end{align*}
such that $\Nm_\pi \circ \pi^*$ is the $n$-th tensor power.
\end{prop}
\begin{proof}
The inverse image for line bundles is a homomorphism since it commutes obviously with tensor products. The other results follow from (\ref{normpr}), using the fact that tensor product of line bundles corresponds to multiplication at the level of defining cocycles.
\end{proof}

\begin{prop}\label{DetFormula}
Let $\pi: X \rightarrow Y$ be a finite, flat morphism between curves. For any invertible $\Oo_X$-modules $\Ll$, we have
\begin{align}
\Nm_\pi(\Ll) \simeq \det(\pi_* \Ll) \otimes \det(\pi_* \Oo_X)^{-1}.
\end{align}
\end{prop}
\begin{proof}
See \cite[Corollary 3.12]{HP}.
\end{proof}

As we show in the next proposition, the direct image for Cartier divisors and the Norm map for line bundles are related, as well as the inverse image maps. Recall that, on any curve $X$, the Picard group is canonically isomorphic to the group of Cartier divisors modulo the subgroup $\Prin(X)$ of principal divisors. This gives rise to a canonical quotient of groups: 
\begin{align*}
q_X: \CDiv(X) &\longrightarrow \CDiv(X)/\Prin(X) = \Pic(X) \\
D &\longmapsto \Ii_D.
\end{align*}
\begin{prop}
Let $\pi: X \rightarrow Y$ be a finite, flat morphism between curves. The direct and inverse image for Cartier divisors are compatible via the quotient maps $q_X$ and $q_Y$ respectively with the Norm and the inverse image map for line bundles; i.e. the following diagrams of groups are commutative: \[
\begin{tikzcd}
\CDiv(X) \arrow{r}{\pi_*} \arrow{d}{q_X}  & \CDiv(Y) \arrow{d}{q_Y} \\
\Pic(X) \arrow{r}{\Nm_\pi} & \Pic(Y)
\end{tikzcd}
\hspace{2em}
\begin{tikzcd}
\CDiv(Y) \arrow{r}{\pi^*} \arrow{d}{q_Y}  & \CDiv(X)
\arrow{d}{q_X} \\
\Pic(Y) \arrow{r}{\pi^*} & \Pic(X).
\end{tikzcd}
\]
\end{prop}
\begin{proof}
The commutativity of the first diagram follows from correspondence between  $D$ and $\Ii_D$ and the definitions of $\pi_*$ and $\Nm_\pi$. The second diagram is commutative as consequence of \cite[Picture 21.4.2.1]{EgaIV4} together with \cite[Proposition 21.4.5]{EgaIV4}.
\end{proof}

\vspace{1em}

We are finally interested in reviewing the Norm map for families of line bundles.\footnote{For the sake of completeness, we note that a complete discussion regarding families of Cartier divisors and their direct and inverse images is done in \cite[§21.15]{EgaIV4}. } To study families of line bundles, we \textbf{assume that $X$ and $Y$ are projective curves over a base field $k$} and that $\pi: X \rightarrow Y$ is a finite, flat morphism of degree $n$ defined over $k$.

\begin{defi}\label{JacobianDef}
Let $d \in \Z$ be an integer number. The \textit{Jacobian scheme} of degree $d$ on $X$ is the algebraic scheme $J^d(X)$ representing the sheafification of the functor that associates to any $k$-scheme $T$ the set of isomorphism classes of line bundles of degree $d$ on $X \times_k T$. The union of the Jacobians of all degrees is denoted as $J(X)$.
\end{defi}

\begin{deflemma}\label{NormOnJacDef}
Let  $\pi: X \rightarrow Y$ be a finite, flat morphism of projective curves over a base field $k$. For any $k$-scheme $T$, the \textit{Norm map for line bundles associated to $\pi$} is defined  on the $T$-valued points of the Jacobian of $X$ as:
\begin{align*}
\Nm_\pi(T): \hspace{2em} J(X)(T) &\longrightarrow J(Y)(T) \\
\Ll &\longmapsto \det\left( \pi_{T,*} (\Ll) \right) \otimes \det\left( \pi_{T,*}\Oo_{X\times_k T} \right)^{-1}.
\end{align*}
where $\pi_T: X \times_k T \rightarrow Y \times_k T$ is induced by pullback from $\pi$.
\end{deflemma}
\begin{proof}
We need to check that the map is well defined. Since $\Ll$ is a line bundle on $X \times_k T$, its pushforward $\pi_{T,*} (\Ll)$ is a locally free sheaf of rank $n$ on $Y \times_k T$, as for $\pi_{T,*}\Oo_{X\times_k T}$. Taking determinants produces line bundles, so $\det\left( \pi_{T,*} (\Ll) \right) \otimes \det\left( \pi_{T,*}\Oo_{X\times_k T} \right)^{-1}$ is a well-defined line bundle on $Y \times_k T$. The whole construction is functorial, so the above definition gives a well-defined morphism of schemes \[
\Nm_\pi: \hspace{2em} J(X) \longrightarrow J(Y).
\]\end{proof}

\section{Review of Jacobian stacks}\label{SectionReviewJacobians}
In this section, we first review and compare different classes of coherent sheaves with well-defined rank and degree. The aim is to compare the different meanings for a (torsion-free) sheaf to be ``of rank 1''. In particular, moving from the simplest case of line bundles, we distinguish among generalized line bundles, torsion-free rank-1 sheaves and torsion-free sheaves of polarized rank 1. Then, we introduce and compare the corresponding moduli stacks, enlarging the Jacobian scheme.

\subsection{Rank and degree of coherent sheaves}

Let $X$ be any Noetherian scheme of pure dimension $1$. In this subsection, we recall and compare different classes of coherent sheaves with well-defined rank and degree. We start by recalling the following definitions.

\begin{defi}
The \textit{support} of $\Ff$ is the closed set \[
\Supp(\Ff) = \{ x \in X| \Ff_x \neq 0 \} \subseteq X.
\] The \textit{dimension} of $\Ff$ is the dimension of its support and is denoted $\dim(\Ff)$.
\end{defi}

\begin{defi}
The \textit{torsion subsheaf} of $\Ff$ is the maximal subsheaf  $T(\Ff) \subseteq \Ff$ of dimension 0. If $\Ff = T(\Ff)$, we say that $\Ff$ is a \textit{torsion sheaf}. If $T(\Ff)=0$ we say that $\Ff$ is \textit{torsion-free}.
\end{defi}
Equivalently, $\Ff$ is torsion-free if $\dim(\Ee)=1$ for any non-trivial coherent subsheaf $\Ee \subseteq \Ff$, i.e. $\Ff$ is pure of dimension 1. For any sheaf $\Ff$, the quotient sheaf $\Ff/T(\Ff)$ is either zero or torsion-free.

\vspace{1em}
A first class of coherent sheaves with well-defined rank, after vector bundles, is given by generalized vector bundles.
\begin{defi}
$\Ff$ is a \textit{generalized vector bundle} if it is torsion-free and there exists a positive integer $r$ such that $\Ff_\xi \simeq \Oo_\xi^{\oplus r}$ for any generic point $\xi$ of $X$, where $\Ff_\xi$ denotes the stalk of $\Ff$ at $\xi$. The integer $r$ is called the \textit{rank} of the generalized vector bundle $\Ff$. A generalized vector bundle of rank $1$ is also called a \textit{generalized line bundle}.
\end{defi}
Clearly, a vector bundle of rank $r$ is also a generalized vector bundle of the same rank.
\begin{rmk}
Generalized line bundles were introduced in \cite{BayEis} in order to study limit linear series on a ribbon. As pointed out in \cite{EisGreen}, generalized line bundles are strictly related to generalized divisors introduced indipendently by Hartshone in \cite{Har86} (followed by \cite{Har94} and \cite{Har07}). See also Section \ref{SectionReviewDivisors} for details. The generalization to any rank $r$ is due to \cite{Savarese}.
\end{rmk}

\vspace{1em}

We introduce now the notion of rank for any coherent sheaf. Denote with $X_1, \dots, X_s$ the irreducible components of $X$ and with $\xi_i$ the generic point of each irreducible component $X_i$. Denote with $C_i = X_{i,red}$ the reduced subscheme underlying $X_i$ for each $i$.
\begin{defi}
The \textit{multiplicity} of $C_i$ in $X$ is defined as the positive integer: \[
\mult_X(C_i) := \ell_{\Oo_{X, \xi_i}}(\Oo_{X, \xi_i}).
\]
where $\ell_{\Oo_{X, \xi_i}}$ denotes the length as $\Oo_{X, \xi_i}$-module.
\end{defi}

\begin{defi}\label{RankDef}(Rank and multirank of a coherent sheaf)
The \textit{rank} of $\Ff$ on $X_i$  is defined as the rational number \[
\rk_{X_i}(\Ff) = \frac{\ell_{\Oo_{X, \xi_i}}(\Ff_{\xi_i}) }{  \ell_{\Oo_{X, \xi_i}}(\Oo_{X, \xi_i}) }.
\]

The \textit{multirank} of $\Ff$ on $X$ is the $n$-uple $\rvect=(r_1, \dots, r_n)$ where $r_i$ is the rank  of $\Ff$ at $X_i$. The sheaf $\Ff$ has \textit{rank} $r$ on $X$ if it has rank $r$ at each $X_i$.
\end{defi}

Note that, since length of modules is additive in short exact sequences, the rank and multirank of coherent sheaves are additive in short exact sequences too.

\begin{prop}
A generalized vector bundle of rank $r$ on $X$ has rank $r$ as a coherent sheaf. If $X$ is reduced, any torsion-free sheaf of rank $r$ is a generalized vector bundle of the same rank.
\end{prop}
\begin{proof}
The first statement follows from \ref{RankDef} and the fact that isomorphic modules have the same length. For the second assertion, note that on an integral curve $X$ there exists for any torsion-free sheaf $\Ff$ an open dense subset $U \subset X$ such that $\Ff_{|U}$ is locally free, hence the stalk of $\Ff$ at the generic point of $X$ is a free $\Oo_{X, \xi}$-module of rank $\rk_X(\Ff)$. If $X$ is reducible, the assertion follows by considering any irreducible component.
\end{proof}

\begin{rmk}
There are in literature other notions for the rank of a coherent sheaf on a irreducible component. A classical notion is the \textit{reduced rank} of $\Ff$ on $X_i$: \[
\rr_{X_i}(\Ff) = \dim_{\kappa(\xi_i)}(\Ff_{\xi_i} \otimes_{\Oo_{X,\xi}} \kappa(\xi_i)).
\]
This definition computes the rank of the sheaf restricted to its reduced support; it agrees with the other definitions for generalized vector bundles. However, there are cases when the reduced rank of a torsion-free sheaf differs from its rank, as it happens for quasi-locally free sheaves on a ribbon (see \cite[§1.4]{Savarese}). They also provide examples of torsion-free sheaves which are not generalized vector bundles, even if their rank is well defined.

The definition of $\rk_H(\Ff)$ comes at least from \cite[Définition 1.2]{Schaub98}, where it is given in the context of projective $k$-schemes without embedded points. The notion of \textit{generalized rank} introduced in \cite{Drez04} and \cite{Savarese} is essentially equivalent. Finally, a common definition in projective algebraic geometry is the polarized rank, that we will discuss later on in the present section.
\end{rmk}

In order to introduce the degree of coherent sheaves, from now on we \textbf{ assume that $X$ is projective curve over a base field $k$}. Recall that the Euler characteristic of a coherent sheaf $\Ff$ is \[
\chi(\Ff) :=  \sum (-)^i \dim_k H^i(X, \Ff).
\]

\begin{defi}\label{DegreeDef}(Degree of a coherent sheaf)
Suppose that $\Ff$ has rank $r$ on $X$. Then, the \textit{degree} of $\Ff$ on $X$ is defined as the fractional number: \[
\deg \Ff = \chi(X, \Ff) - r \chi(X, \Oo_X).
\]
\end{defi}
\begin{rmk}
If the rank of $\Ff$ is integer, the degree is integer as well. If $X$ is integral, the degree of a sheaf is an integer and coincides with the classical notion of degree for sheaves on integral curves. 
\end{rmk}
The following technical lemma is very useful.

\begin{lemma}\label{ChiOfTensor}
Assume that $X$ is irreducible. Let $\Ff$ be a coherent sheaf of rank $r$ on $X$ and let $\Ee$ be a locally free sheaf of rank $n$ on $X$. Then \[
\chi(X, \Ff \otimes \Ee) = r\deg(\Ee) + n\chi(X, \Ff).
\]
\end{lemma}
\begin{proof}
The proof is inspired by \cite[\href{https://stacks.math.columbia.edu/tag/0AYV}{Tag 0AYV}]{stacks-project} and uses devissage for coherent sheaves as stated in \cite[\href{https://stacks.math.columbia.edu/tag/01YI}{Tag 01YI}]{stacks-project}. Let $\mathcal{P}$ be the property of coherent sheaves $\Ff$ on $X$ expressing that the formula of the Lemma holds. By additivity of rank and Euler characteristic in short exact sequences, $\mathcal{P}$ satisfy the two-out-of-three property. The integral subschemes $Z$ of $X$ are the reduced subscheme $C$ with support equal to $X$, and the closed points of $X$. For $Z=C$, the formula of the Lemma for $\Oo_{C}$ is: \[
\chi(X, \Ee \otimes \Oo_{C}) =\frac{\deg (\Ee)}{\mult_X C} + n \chi(X, \Oo_{C}). \]
This is true by definition of degree of $\Ee \otimes \Oo_{C} = \Ee_{|C}$ on $C$ and the fact that $\deg(\Ee)=\mult_X (C)\deg(\Ee_{|C})$. Then $\mathcal{P}(\Oo_C)$ holds. If $i: Z \hookrightarrow X$ is a closed point, the formula of the Lemma is true for $i_*\Oo_Z$ since it is a torsion sheaf of rank 0 and $\chi(X, \Ee \otimes i_*\Oo_Z) = n\chi(X, i_*\Oo_Z)$. Then $\mathcal{P}(i_*\Oo_Z)$ holds.
\end{proof}

\vspace{1em}
Let $H$ be any polarization of $X$; then, the notion of polarized rank and degree of sheaves can be defined. Denote with $H_{|C_i}$ the restriction of $H$ to the closed subscheme $C_i \subseteq X$. Since $H$ is locally free of rank $1$, the degree of $H$ on $X$ is related to the degrees of each  $H_{|C_i}$ on $C_i$ by the formula: \[
\deg H = \sum_{i=1}^s \mult_X( C_i) \deg H_{|C_i}.
\]

\begin{defi}\label{PolarizedRankDef}(Polarized rank and degree of a torsion-free sheaf)
Let $H$ be any polarization of $X$ with degree 	$\deg H = \delta$ and let $\Ff$ be a torsion-free sheaf on $X$. The \textit{polarized rank and 	degree} of $\Ff$ are the rational numbers $r_H(\Ff)$ and $d_H(\Ff)$ determined by the Hilbert polinomial of $\Ff$ with respect to $H$: \[
P(\Ff, n, H) := \chi(\Ff \otimes \Oo_X(nH)) = \delta r_H(\Ff) n + d_H(\Ff) + r_H(\Ff) \chi(\Oo_X). 
\]
\end{defi}

The polarized rank and degree of a sheaf depend strictly on the degrees of the restrictions $H_{|C_i}$, as the following theorem shows.

\begin{thm}\label{PolarizedRankVsRank}
Let $H$ be a polarization of $X$ and let $\Ff$ be a torsion-free sheaf on $X$. The polarized rank of $\Ff$ is related to the multirank of $\Ff$ by the formula \[
r_H(\Ff)=\frac{\sum_{i=1}^s \rk_{X_i}(\Ff) \mult_X (C_i) \deg H_{|C_i} }{ \sum_{i=1}^s \mult_X( C_i) \deg H_{|C_i} }.
\]
\end{thm}
\begin{proof}
Since the Hilbert polynomial of $\Ff$ with respect to $H$ has degree $1$, its leading term can be computed in terms of Euler characteristic as \[
\chi(X, \Ff \otimes \Oo_X(nH))  - \chi(X, \Ff).
\]
Since $\delta=\sum_{i=1}^n \mult_X( C_i) \deg_{C_i}H $ is the degree of $H$ on $X$, by Definition \ref{PolarizedRankDef} we have to prove: \[
\chi(X, \Ff \otimes \Oo_X(nH))  - \chi(X, \Ff) = n \sum_{i=1}^s \rk_{X_i}(\Ff) \mult_X (C_i) \deg_{C_i}H.
\]
First, we reduce to the case of $X$ irreducible. Consider the exact sequence: \[
0 \rightarrow \Ff \rightarrow \bigoplus_i \Ff_{|X_i} \rightarrow T \rightarrow 0
\]
where $T$ is a torsion sheaf supported only at the intersections of the irreducible components. Tensoring by $nH$, we obtain another exact sequence: \[
0 \rightarrow \Ff  \otimes \Oo_X(nH)  \rightarrow \bigoplus_i (\Ff  \otimes \Oo_X(nH) )_{|X_i} \rightarrow T \otimes \Oo_X(nH) \rightarrow 0.
\]
By additivity of the Euler characteristic with respect to exact sequences we have:
\begin{align*}
\sum_i &\left( \chi(X_i, (\Ff \otimes \Oo_X(nH))_{|X_i})  - \chi(X_i, \Ff_{|X_i}) \right)=\\
&=\chi(X, \Ff \otimes \Oo_X(nH))  - \chi(X, \Ff )+\chi(X, T \otimes \Oo_X(nH))  - \chi(X,T ).
\end{align*}
Since $T$ has dimension 0, the characteristics $\chi(X, T \otimes \Oo_X(nH))$ and $\chi(X,T )$ are the same, hence: \[
\chi(X, \Ff \otimes \Oo_X(nH))  - \chi(X, \Ff) = \sum_i \left( \chi(X_i, (\Ff \otimes \Oo_X(nH))_{|X_i})  - \chi(X_i, \Ff_{|X_i}) \right).
\]
We are then reduced to prove the statement on each irreducible component $X_i$. In other words we have to prove that, if $X$ is a irreducible (possibly non-reduced) curve with reduced structure $C$, then: \[
\chi(X, \Ff \otimes \Oo_X(nH))  - \chi(X, \Ff) = n \deg_X (H) \rk_X (\Ff).
\]
This is exactly the content of Lemma \ref{ChiOfTensor}, with $\Ee=\Oo_X(nH)$.
\end{proof}

\begin{cor}
Let $\Ff$ be any torsion-free sheaf with well-defined rank. Then, its polarized rank $r_H(\Ff)$ and degree $d_H(\Ff)$ are equal respectively to $\rk_X(\Ff)$ and $\deg_X(\Ff)$ for any polarization $H$ on $X$. f $X$ is irreducible, the rank and degree of a sheaf coincide with the polarized counterparts, for any polarization.
\end{cor}
\begin{proof}
The statement about the rank follows from Theorem \ref{PolarizedRankVsRank}, where $r_i = r$ for each irreducible component $X_i$. To prove the statement about the degree, note that for any polarization $H$: \[
d_H(\Ff) = \chi(X, \Ff) - r_H(\Ff)\chi(X, \Oo_X).
\]
If $r_H(\Ff) = \rk(\Ff)$, then $d_H(\Ff)= \deg(\Ff)$ by Definition \ref{DegreeDef}. The last statement follows from the fact that the $\rk_X(\Ff)$ is well-defined for any torsion-free sheaf when $X$ has only one irreducible component.
\end{proof}

\begin{rmk}
For reducible curves, the definitions of polarized rank and degree are the most general. As pointed out in \cite[§2]{Lop}, being of polarized rank $1$ for a torsion-free sheaf does not ensure that the sheaf is supported on the whole curve. Moreover, there are cases of torsion-free sheaves of polarized rank $1$ whose polarized rank on the restrictions to the irreducible components is different from $1$ and not an integer, even for reduced curves.
\end{rmk}

\subsection{Moduli of torsion-free sheaves ``of rank 1''}\label{SectionModuliOfSheaves}
In this subsection, we introduce the moduli spaces for generalized line bundles, for torsion-free rank-1 sheaves and for torsion-free sheaves with polarized rank 1, over a (possibly reducible, non-reduced) projective curve $X$, defined over a base field $k$, with irreducible components $X_1, \dots, X_s$.

\vspace{1em}
Since we are dealing with moduli spaces, we recall first the definition of slope and stability for torsion-free sheaves.

\begin{defi}
Let $H$ be a polarization on $\Ff$ and $\Ff$ be a torsion-free sheaf on $X$ with polarized rank $r_H(\Ff)$ and polarized degree $d_H(\Ff)$.

The \textit{slope} of $\Ff$ with respect to $H$ is defined as the rational number $\mu_H(\Ff)=d_H(\Ff)/r_H(\Ff)$.

A coherent sheaf $\Ff$ on $X$ is $H$-\textit{stable} (respectively $H$-\textit{semi-stable}) if it is torsion-free and for any proper subsheaf $\Ee \subset \Ff$ the equality $\mu_H(\Ee) < \mu_H(\Ff)$ holds (respectively $\leq$).
\end{defi}

We start from the last moduli space, which is the larger and contains all the others.

\begin{deflemma}\label{SimpsonJacobianDef}
Let $H$ be a polarization on $X$ and let $d \in  \Z$ be an integer number. The \textit{Simpson Jacobian stack} of degree $d$ on $X$ is the algebraic stack  $\Jtf(X,H,d)$  such that for any $k$-scheme $T$, $\Jtf(X,d)(T)$ is the groupoid of $T$-flat coherent sheaves on $X \times_k T$ whose fibers over $T$ are torsion-free sheaves  of polarized rank 1 and polarized degree $d$ on $X \simeq X \times_k \{t\}$ with respect to $H$.

The \textit{Simpson Jacobian} of degree $d$ is the projective scheme $\JtfSch(X,H,d)$ which is the good moduli space representing S-equivalence classes of $H$-semistable torsion-free sheaves of polarized rank 1 and polarized degree $d$ on $X$.

The union of the Simpson Jacobians of all degrees is denoted $\Jtf(X,H)$ (resp. $\JtfSch(X,H)$).
\end{deflemma}
\begin{proof}
The algebraicity of $\Jtf(X,H,d)$ follows from the algebraicity of the stack of coherent sheaves on $X$ (see \cite[2.1]{Lieb} and \cite[Theorem 7.20]{CMW}). The moduli space of semistable sheaves with fixed polarized rank and degree was constructed by Simpson in \cite{SimpI}.
\end{proof}

In the following proposition, we show that the multirank of sheaves splits $\Jtf(X,H,d)$ in unions of connected components.

\begin{lemma}\label{CoverForJtf}
Let $H$ be a fixed polarization on $X$. For any  $\rvect=(r_1, \dots, r_s) \in \mathbb{Q}_{\geq 0}^s$, let $W_\rvect \subseteq \Jtf(X,H,d)$ be the substack parametrizing torsion-free sheaves with multirank equal to $\rvect$. Then $\{W_\rvect | W_\rvect \neq \emptyset \}_{\rvect \in \mathbb{Q}_{\geq 0}^s}$ is a finite collection of pairwise disjoint substacks of $\Jtf(X,H,d)$ that covers the whole space. Moreover, each $W_\rvect$ is open and closed, hence union of connected components of $\Jtf(X,H,d)$.
\end{lemma}
\begin{proof}
The multi-rank of a sheaf is well-defined, hence the substacks $W_\rvect$ are pairwise disjoint.

Let $\Ff$ be a sheaf with polarized rank $r_H$ equal to $1$. By Theorem \ref{PolarizedRankVsRank} its multirank $\rvect$ must satisfy: \[
1 = r_H(\Ff) = \frac{\sum_{i=1}^s r_i(t) \mult_X (C_i) \deg_{C_i}H }{ \sum_{i=1}^n \mult_X( C_i) \deg_{C_i}H }.
\]
Since each $r_i$ is a non-negative fraction with integer numerator and denominator equal to $\mult_X( C_i)$, the possible values for $\rvect$ form a finite subset of $\mathbb{Q}_{\geq 0}^s$, once that $H$ and $r_H$ are fixed. Then, only a finite number of $W_\rvect$ is non-empty, hence  $\{W_\rvect | W_\rvect \neq \emptyset \}$ is a finite collection.

We claim that each $W_\rvect$ is closed; since the collection is finite and the $W_\rvect$'s are pairwise disjoint, this implies that each $W_\rvect$ is open and closed, and hence union of connected components.

To prove that $W_\rvect$ is closed, let $T$ be any $k$-scheme and let $\Ff$  be any family of torsion-free sheaves of polarized rank $1$ on $X \times_k T \xrightarrow{\pi} T$. Suppose that $\Ff$ has generically multirank $\rvect$ on $X$, meaning that $\rk_{X_i} (\Ff_{|\pi^{-1}(\eta)}) = r_i$ for each irreducible component $X_i$ of $X$ and any generic point $\eta$  of $T$.
Let $t$ be any point in the closure of $\{\eta\}$. Since the length is upper-semicontinuous, the rank of $\Ff_{|\pi^{-1}(t)}$ at each $X_i$ is a rational number $r_i(t)$ greater or equal than $r_i$. On the other hand, $\Ff_{|\pi^{-1}(t)}$ have polarized rank $1$; hence, thanks to Theorem \ref{PolarizedRankVsRank}, we can write: \begin{align*}
1 = r_H(\Ff) &= \frac{\sum_{i=1}^s r_i(t) \mult_X (C_i) \deg_{C_i}H }{ \sum_{i=1}^n \mult_X( C_i) \deg_{C_i}H } \geq \\
&\geq \frac{\sum_{i=1}^s  r_i \mult_X (C_i) \deg_{C_i}H }{ \sum_{i=1}^n \mult_X( C_i) \deg_{C_i}H }  = 1.
\end{align*}
By difference we obtain: \[
\frac{\sum_{i=1}^s (r_i(t)-r_i) \mult_X (C_i) \deg_{C_i}H }{ \sum_{i=1}^n \mult_X( C_i) \deg_{C_i}H } = 0.
\]
Since all the terms of the expression are non-negative, the only possibility is that $r_i(t)=r_i$ for each $i$. We conclude that the whole family $\Ff$ belongs to $W_\rvect$. 
\end{proof}

We come now to torsion-free sheaves with well-defined rank 1 on $X$.
\begin{deflemma}\label{CompJacobianDef}
Let $d \in \Z$ be an integer number. The \textit{compactified Jacobian stack} of degree $d$ on $X$ is the algebraic stack $\Jbar(X,d)$ such that for any $k$-scheme $T$, $\Jbar(X,d)(T)$ is the groupoid of $T$-flat coherent sheaves on $X \times_k T$ whose fibers over $T$ are torsion-free sheaves of rank 1 and degree $d$ on $X \simeq X \times_k \{t\}$.

Let $H$ be a polarization on $X$. The \textit{compactified Jacobian scheme} of degree $d$ is the good moduli space $\JbarSch(X,H,d)$ representing S-equivalence classes of $H$-semistable torsion-free sheaves of rank 1 and degree $d$ on $X$.

The union of the compactified Jacobians of all degrees is denoted $\Jbar(X)$ (resp. $\JbarSch(X,H)$).
\end{deflemma}
\begin{proof}
The property of having rank $1$ on $X$ implies that the polarized rank is equal to $1$ for any polarization; moreover, the degree is well defined and equal to the polarized degree for any polarization. Fix any polarization $H$ on $X$. Then, $\Jbar(X,d)$ can be seen as the open and closed substack $W_{(1, \dots, 1)}$ of $\Jtf(X, H, d)$ (Lemma \ref{CoverForJtf}), and hence is algebraic. The same argument holds for $\JbarSch(X,H,d)$, with the difference that the notion of semi-stability strictly depends on the choice of $H$.
\end{proof}

The following proposition shows that $\JbarSch(X,H,d)$ is actually a compactification of $J^d(X)$.

\begin{prop}
For any polarization $H$ on $X$, $\JbarSch(X,H,d)$ is a projective scheme,  union of connected components of $\JtfSch(X,H,d)$, and contains $J^d(X)$ as an open subscheme.
\end{prop}
\begin{proof}
Since $\Jbar(X,d)$ is an open and closed substack of $\Jtf(X,H,d)$, then $\JbarSch(X,H,d)$ is an open and closed subscheme of $\JtfSch(X, H, d)$ and union of connected components. It is non-empty, as it contains the structure sheaf $\Oo_X$. Moreover $\JtfSch(X,H,d)$ is projective, hence $\JbarSch(X,H,d)$ is projective too. Finally, the condition of being locally free is open and implies both stability and rank 1; hence, $J^d(X)$ is an open subscheme of $\JbarSch(X,H,d)$.
\end{proof}

We come now to the last moduli space, which parametrizes generalized line bundles.

\begin{deflemma}\label{GenJacobianDef}
Let $d \in \Z$ be an integer number. The \textit{generalized Jacobian stack} of degree $d$ on $X$ is the algebraic stack $\Jgen(X,d)$ such that, for any $k$-scheme $T$, $\Jgen(X,d)(T)$ is the groupoid of $T$-flat coherent sheaves on $X \times_k T$ whose fibers over $T$ are generalized line bundles of degree $d$ on $X \simeq X \times_k \{t\} $.

Let $H$ be a polarization on $X$. The \textit{generalized Jacobian scheme} of degree $d$ is the good moduli space $\JgenSch(X,H,d)$ representing S-equivalence classes of $H$-semistable generalized line bundles of degree $d$ on $X$.

The union of the generalized Jacobians of all degrees is denoted $\Jgen(X)$ (resp. $\JgenSch(X,H)$).
\end{deflemma}
\begin{proof}
Generalized line bundles are torsion-free sheaves of rank 1; moreover, the condition of being generically locally free is open, hence $\Jgen(X,d)$ can be seen as an open substack of $\Jbar(X,d)$. In particular, $\Jgen(X,d)$ is an algebraic stack. Similarly, $\JgenSch(X,H,d)$ is an open subscheme of $\JbarSch(X,H,d)$.
\end{proof}

\vspace{1em}

To sum up, we defined a number of moduli spaces satisfying the following chain of inclusions: \[
J^d(X) \stackrel{(1)}{\subseteq} \Jgen(X,d) \stackrel{(2)}{\subseteq} \Jbar(X,d) \stackrel{(3)}{\subseteq} \Jtf(X,H,d),
\]
and: \[
J^d(X) \stackrel{(1')}{\subseteq} \JgenSch(X,H,d) \stackrel{(2')}{\subseteq} \JbarSch(X,H,d) \stackrel{(3')}{\subseteq} \JtfSch(X,H,d).
\]
\begin{rmk}\label{JacobianInclusions}
Inclusions $(1)$ and $(2)$ (resp. $(1')$ and $(2')$) above are open embeddings; inclusion $(3)$ (resp. $(3')$) is an open and closed embedding. In general they are strict and not dense, in the sense that for any of them there exists a (possibly non-reduced or reducible) curve $X$ such that the closure of the smaller space in the bigger one is contained strictly. Anyway, when $X$ satisfies additional conditions, some of the inclusions above are actually equalities.
\begin{itemize}
	\item If $X$ is irreducible, inclusion $(3)$ (resp. $(3)'$) is an equality.
	\item If $X$ is reduced, inclusion $(2)$ (resp. $(2')$) is an equality.
	\item  If $X$ is integral and smooth, inclusions $(1)$, $(2)$ and $(3)$ (resp. $(1')$, $(2')$ and $(3')$) are equalities.
\end{itemize}
\end{rmk}

\section{Review of generalized divisors}\label{SectionReviewDivisors}

In this section, we review R. Hartshorne's theory  of generalized divisors in the case of curves, in connection with the theory of torsion-free sheaves presented in Section \ref{SectionReviewJacobians}. We show that generalized divisors with given degree $d$, up to linear equivalence, correspond to generalized line bundles of degree $-d$; this fact is reflected in the definition of the Abel map given at the end of this section, comparing the moduli space of effective generalized divisors with the generalized Jacobian.

\begin{rmk}
The theory of generalized divisors is developed by Hartshorne in his papers \cite{Har86}, \cite{Har94} and \cite{Har07}, in order to generalize the notion of Cartier divisor on schemes satisfying condition $S_2$ of Serre. Since we are dealing with schemes of dimension 1, the condition $S_2$ of Serre coincides with the condition $S_1$, which in turn coincides with the fact of not having embedded components (i.e. embedded points). In particular, any scheme of pure dimension 1 is also $S_1$, and hence $S_2$. Similarly, a coherent sheaf on a curve satisfies condition $S_2$ if and only if it is torsion-free.
\end{rmk}

\subsection{Generalized divisors}

\begin{defi}
Let $X$ be a curve and let $\Kk_X$ be the sheaf of total quotient rings on $X$. A \textit{generalized divisor} on $X$ is a nondegenerate fractional ideal of $\Oo_X$-modules, i.e. a subsheaf $\Ii \subseteq \Kk_X$ that is a coherent sheaf of $\Oo_X$-modules and such that $\Ii_\eta = \Kk_{X, \eta}$ for any generic point $\eta \in X$. It is \textit{effective} if $\Ii \subseteq \Oo_X$. It is \textit{Cartier} if $\Ii$ is an invertible $\Oo_X$-module, or equivalently locally principal. It is  \textit{principal} if $\Ii = \Oo_X \cdot f$ (also denoted $(f)$) for some global section $f \in \Gamma(X, \Kk^*_X)$. 
\end{defi}

The set of generalized divisors on $X$ is denoted with $\GDiv(X)$, the subset of Cartier divisors with $\CDiv(X)$ and the set of principal divisors with $\Prin(X)$. 

Using the usual notion of subscheme associated to a sheaf of ideals, the set $\GDiv^+(X)$ of effective generalized divisors on $X$ is in one-to-one correspondence with the set of closed subschemes $D \subset X$ of pure codimension one (i.e. of dimension zero). With a slight abuse of notation, also for non-effective divisors, we denote with $D$ the generalized divisor and we refer to $\Ii$ (or $\Ii_D$) as the \textit{fractional ideal} of $D$ (also called \textit{defining ideal} of $D$, or \textit{ideal sheaf} of $D$ if $D$ is effective).

\vspace{1em} 

Let $D_1, D_2$ be generalized divisors on $X$, with fractional ideals $\Ii_1, \Ii_2$. The \textit{sum} $D_1 + D_2$ is the generalized divisor associated to the fractional ideal $\Ii_1 \cdot \Ii_2 \subseteq \Kk_X $. The sum is commutative, associative, with neutral element $0$ defined by the trivial ideal $\Oo_X$. The \textit{inverse} of a generalized divisor $D$ associated to $\Ii$ is the generalized divisor $-D$ associated to the fractional ideal $\Ii^{-1}$, i.e. the sheaf which on any open subset $U$ of $X$ is defined as  $\{ f \in \Gamma(U,\Kk_X) \hspace{.3em} | \hspace{.3em} \Ii(U) \cdot f \subseteq \Oo_X(U) \}$.

The inverse operation behaves well with the sum only for Cartier divisors. For any pair of generalized divisors $D,E$ on $X$ with $E$ Cartier, $-(D+E)= (-D)+(-E)$, but $D + (-D)=0$ if and only if $D$ is a Cartier divisor. As a consequence, the set $\GDiv(X)$ of all generalized divisors over $X$ endowed with the sum operation is not a group, but the subset $\CDiv(X)$ of Cartier divisors is. The set $\Prin(X)$ of principal divisors is a subgroup of $\CDiv(X)$ and both the groups act by addition on the set $\GDiv(X)$. 

Tthe following Lemma, inspired by \cite[Proposition 2.11]{Har94}, shows that any generalized divisor is equal to the sum of an effective generalized divisor and the inverse of an effective Cartier divisor, as 

\begin{lemma}\label{DifferenceOfEffective}
Let $X$ be a curve and let $D \in \GDiv(X)$ be any generalized divisor on $X$. Then, there exist an effective generalized divisor $D'$ and an effective Cartier divisor $E$ on $X$ such that $D=D'+(-E)$.
\end{lemma}
\begin{proof}
Cover $X$ by open affines $U_i=\Spec(A_i)$, $i=1, \dots, r$. For each $i$, denote with $I_i$ the fractional ideal of $D$ restricted to $U_i$. This is a finitely generated $A_i$-module, so there exists a nonzero-divisor $f_i \in A_i$ such that $f \cdot I_i \subseteq A$. Let $Y_i \subset U_i$ be the closed subscheme defined by $f_i$, which is an effective Cartier divisor of $U_i$; after the composition with $U_i \hookrightarrow X$, $Y_i$ becomes an effective Cartier divisor of $X$. Now, the sum of divisors $D+\sum_{i=1}^r Y_i$ is effective as it is sum of effective divisors on each $U_i$. By putting $E= \sum_{i=1}^r Y_i$ and $D'=D+E$ we get the result.
\end{proof}

\vspace{1em} 

Two generalized divisors $D_1, D_2$ over $X$ are \textit{linearly equivalent} (written $D_1 \sim D_2$) if there is a divisor $(f) \in \Prin(X)$ such that $D_1 + (f) = D_2$. We define the \textit{generalized Picard of X}  as the set of divisors on $X$ modulo linear equivalence: \[
\GPic(X) = \GDiv(X)/\Prin(X).
\]  
We have the following commutative diagram of sets, where the vertical maps are quotients by $\Prin(X)$:
\begin{center}
	\begin{tikzcd}[column sep=.7em]
	\CDiv(X) \arrow[twoheadrightarrow]{d} & \subseteq & \GDiv(X)\arrow[twoheadrightarrow]{d} \\
	\Pic(X)  & \subseteq & \GPic(X)
\end{tikzcd}

\end{center}
Taking inverse and sums preserve linear equivalence, so the two operations are well defined also on $\GPic(X)$ and the subset $\Pic(X) \subseteq \GPic(X) $ is a group that acts on $\GPic(X)$ by addition. For a curve $X$, the condition $\GDiv(X) = \CDiv(X)$ is also equivalent to $\GPic(X)= \Pic(X)$, which is also equivalent to the curve $X$ being smooth.

\vspace{1em}
As in the case of Cartier divisors, also the set $\GPic(X)$  has an alternative description in terms of abstract sheaves. First, we need a definition.

The set $\GPic(X)$  has an alternative description in terms of generalized line bundles. Let $D$ be a generalized divisor on $X$; then, its fractional ideal $\Ii$ is a generalized line bundle. If  $D'$ is another generalized line bundle, it is linearly equivalent to $D$ if and only if its fractional ideal $\Ii'$ is a generalized line bundle isomorphic to $X$ as an $\Oo_X$-module. Viceversa, any generalized line bundle $\Ff$ on $X$ is isomorphic to the fractional ideal of some generalized divisor $D$ \cite[Proposition 2.4]{Har07}. Then, \textit{$\GPic(X)$ can be also defined as the set of generalized line bundles of $X$, up to isomorphism}.
Under this description, classes of Cartier divisors correspond to isomorphism classes of line bundles, and the operations of sum by a Cartier divisor and inverse of a divisor are replaced with tensor product and taking dual of the corresponding sheaves.

\subsection{Degree of generalized divisors}

We now \textbf{assume that $X$ is a projective curve over a base field $k$}. In this case, the notion of \textit{degree of a divisor} can be introduced. 

\begin{deflemma}\label{DegreeGenDivDef}
Let $D \in \GDiv^+(X)$ be an effective generalized divisor on $X$ with ideal sheaf $\Ii _D \subseteq \Oo_X$, and let $x \in X$ be any point of $X$ in codimension 1. We define the \textit{degree of $D$ at $x$} as the non-negative integer: \[
\deg_x(D) = \ell_{\Oo_{X,x}}\big( \Oo_{X,x}/\Ii_{D,x} \big)[\kappa(x):k].
\] 

The degree of any generalized divisor $D \in \GDiv(X)$ at $x$ is defined as $\deg_x(D)=\deg_x(E)-\deg_x(F)$, where $D=E-F$ with $E,F$ effective generalized divisors and $F$ Cartier by Lemma \ref{DifferenceOfEffective}. 

The degree of $D$ on $X$ (also denoted as $\deg(D)$ when there is no ambiguity) is equal to the sum of the degrees of $D$ at all points of $X$ in codimension 1: \[
\deg_X(D)= \sum_{x \textrm{ cod 1}} \deg_x(D).
\]
\end{deflemma}
\begin{proof}
First, note that the local ring $\Oo_{X,x}/\Ii_{D,x} $ in nonzero if and only if $x$ is in the support of $D$. In such case the local ring has Krull dimension 0, so it is Artinian and hence has finite length.

The definition of degree at a point of any generalized divisor is well-given thanks to \cite[Lemma 2.17]{Har94}, and $\deg_x(-D) = -\deg_x(D)$.

Finally, the supporty of $D$ contains only finitely many points, hence the degree of $D$ on $X$ is a well-defined integer.
\end{proof}

The degree of a generalized divisor is strictly related to the degree of its fractional ideal, as the following proposition shows.

\begin{prop}
Let $D$ be a generalized divisor on $X$ with fractional ideal $\Ii_D \subseteq \Kk_X$. Then, the degree of $\Ii_D$ as a torsion-free sheaf is well defined and equal to $-\deg_X(D)$. Moreover, linearly equivalent divisors have the same degree on $X$.
\end{prop}
\begin{proof}
The fractional ideal $\Ii_D$ is a generalized line bundle, hence it has rank $1$ on $X$ and well-defined degree as a torsion-free sheaf. In order to compute $\deg(\Ii_D)$ consider the short exact sequence \[
0 \rightarrow \Ii_D \rightarrow \Oo_X \rightarrow \Oo_X/\Ii_D \rightarrow 0.
\]
By definition of degree, $\deg(\Ii_D) = \chi(\Ii_D) - \chi(\Oo_X)$. Then, by additivity of Euler characteristic with respect to exact sequences, we deduce that $\deg(\Ii_D) = - \chi(\Oo_X/\Ii_D)$. Since $\Oo_X/\Ii_D$ has dimension 0, its Euler characteristic is equal to $h^0(\Oo_X/\Ii_D):=\dim_k H^0(X, \Oo_X/\Ii_D)$. The support of $\Oo_D=\Oo_X/\Ii_D$ is by definition the support of $D$ and $H^0(X, \Oo_X/\Ii_D) = \oplus_{x \in \Supp D} \Oo_{D,x}$. Then we compute: \begin{align*}
\deg(\Ii_D) &=  -\chi(\Oo_X/\Ii_D)=-\sum_{x \in \Supp D} \dim_k \Oo_{D,x} = \\
&= -\sum_{x \in \Supp D} \dim_{\kappa(x)} \big( \Oo_{D,x} \big) [\kappa(x):k] = \\
&= -\sum_{x \in \Supp D} \ell_{\Oo_X,x} \big( \Oo_{D,x} \big) [\kappa(x):k] = \\
&= - \deg_X(D).
\end{align*}
Consider now a generalized divisor $D$ with fractional ideal $\Ii_D$ and let $D=E-F$ with $E,F$ effective and $F$ Cartier by Lemma \ref{DifferenceOfEffective}. Since $F$ is effective, consider first the short exact sequence: \begin{align}\label{DegreeExSeq1}
0 \rightarrow \Ii_F \rightarrow \Oo_X \rightarrow T \rightarrow 0
\end{align}
where $T$ is a sheaf of dimension 0. Since $F$ is Cartier, we can tensor by $\Ii_F^{-1}$ to obtain: \begin{align}\label{DegreeExSeq2}
0 \rightarrow \Oo_X \rightarrow \Ii_F^{-1} \rightarrow \Ii_F^{-1} \otimes T \rightarrow 0.
\end{align}
Using addivity of Euler characteristics in the exact sequences \ref{DegreeExSeq1} and \ref{DegreeExSeq2} and the fact that $\chi(\Ii_F^{-1} \otimes T)=\chi(T)$, we see that: \[
\chi(\Ii_F^{-1})-\chi(\Oo_X ) = \chi(\Oo_X)-\chi(\Ii_F).
\]
Also $E$ is effective, hence there is a short exact sequence:\begin{align}\label{DegreeExSeq3}
0 \rightarrow \Ii_E \rightarrow \Oo_X \rightarrow Q \rightarrow 0
\end{align} 
where $Q$ is a sheaf of dimension 0. Tensoring again by $\Ii_F^{-1}$ we obtain: \begin{align}\label{DegreeExSeq4}
0 \rightarrow \Ii_D \rightarrow \Ii_F^{-1} \rightarrow \Ii_F^{-1} \otimes Q \rightarrow 0
\end{align}
where $\Ii_D \simeq \Ii_E \otimes \Ii_F^{-1}$ since $D=E-F$ and $F$ is Cartier. Using addivity of Euler characteristics in the exact sequences \ref{DegreeExSeq3} and \ref{DegreeExSeq4} and the fact that $\chi(\Ii_F^{-1} \otimes Q)=\chi(Q)$, we can compute the degree of $\Ii_D$: \begin{align*}
\deg(\Ii_D) &= \chi(\Ii_D)-\chi(\Oo_X) = \\
&= \chi(\Ii_F^{-1})-\chi(\Ii_F^{-1} \otimes Q)-\chi(\Oo_X) = \\
&= \chi(\Ii_F^{-1})-\chi(Q)-\chi(\Oo_X) = \\
&= \chi(\Ii_F^{-1}) -2\chi(\Oo_X) + \chi(\Ii_E) = \\
&=\chi(\Oo_X) - \chi(\Ii_F) +\chi(\Ii_E) - \chi(\Oo_X).
\end{align*}
Since $E$ and $F$ are effective, we conclude that: \[
\deg(\Ii_D) = - (\deg_X(E)-\deg_X(F)) = -\deg_X(D).
\]
To prove the last statement, let $D$ and $D'$ be linearly equivalent divisors; then, their fractional ideal $\Ii_D$ and $\Ii_{D'}$ are isomorphic. Hence, \[
\deg_X(D) = - \deg(\Ii_D) = - \deg(\Ii_D') = \deg_X(D').
\]
\end{proof}


\subsection{Moduli of generalized divisors and the Abel Map}
We start with considering the moduli space for generalized divisors. Since we are dealing with families of sheaves and related moduli problems, we \textbf{assume that $X$ is a projective curve over a base field $k$}.

\begin{rmk}
In general $\GDiv(X)$ cannot be represented by a geometric object of finite type, even for fixed degrees. Also, it is not easy to give a correct definition for flat families of non-effective generalized divisors. Hence, only families of effective divisors are considered, using the Hilbert scheme.
\end{rmk}

\begin{defi}\label{HilbertDef} The \textit{Hilbert scheme of effective generalized divisors of degree $d$ on $X$} is the Hilbert scheme $\Hilb_X^d$ parametrizing 0-dimensional subschemes of $X$, with Hilbert polynomial equal to a costant integer $d$. Recall that, given a $k$-scheme $T$, a $T$-valued point of $\Hilb_X^d$ is a $T$-flat subscheme $D \subset X \times_k T$ such that $D$ restricted to the fiber over any $t \in T$ is a 0-dimensional subscheme of $X \times_k \{t\} \simeq X$, of degree $d$. In other words, the corresponding ideal sheaf $\Ii_D \subset \Oo_{X\times_k T}$ restricted to any fiber over $T$ defines a  generalized divisor $D_t$ on $X$.

The \textit{Hilbert scheme of effective Cartier divisors of degree $d$} is the open subscheme $\lHilb_X^d \subseteq \Hilb_X^d$ parametrizing subschemes of $X$ of degree $d$ whose ideal sheaf is locally principal.
\end{defi}
Note that $\lHilb_X^d=\Hilb^d_X$ when $X$ is smooth.

Families of effective generalized divisors can be added with families of Cartier divisors, giving rise to a morphism: 
\begin{align*}
\Hilb_X^d \times \lHilb_X^e &\longrightarrow \Hilb_X^{d+e}\\
(D,E) \hspace{2em} &\longmapsto \Zz(\Ii_D \cdot \Ii_E).
\end{align*}

Effective generalized divisors on $X$ up to linear equivalence are generalized line bundles, hence they are parametrized by the generalized Jacobian $\Jgen(X)$.

By Remark \ref{JacobianInclusions}, the Jacobian scheme $J(X)$ of $X$ is contained in $\Jgen(X)$ as an open subscheme. The operations of tensor products and taking inverse are well defined on $J(X)$ and make  $J(X)$ into an algebraic group, acting on $\Jgen(X)$ via tensor product: \begin{align*}
\Jgen(X) \times J(X) &\longrightarrow \Jgen(X) \\
(\Ll, \Ee) \hspace{1.5em}  &\longmapsto \Ll \otimes \Ee.
\end{align*}

\vspace{1em}

Finally, we recall the definition of the geometric Abel map. Let $\Hilb_X = \bigsqcup_{d \geq 0} \Hilb^d_X$ be the Hilbert scheme of effective generalized divisor of any degree on $X$.
\begin{defi}\label{TwistedAbel}
Let $M$ be a line bundle of degree $e$ on $X$. The \textit{($M$-twisted) Abel map} of degree $d$ is defined as:
\begin{align*}
\A_M^d: \hspace{2em} \Hilb_X^d &\longrightarrow \Jgen(X,-d+e) \\
D &\longmapsto \Ii_D \otimes M.
\end{align*}
The $M$-twisted Abel map in any degree is defined similary as a map: \[
\A_M: \hspace{1em} \Hilb_X \longrightarrow \Jgen(X)
\]
\end{defi}

The restriction of $\A_M^d$ to $\lHilb_X^d $ takes values in $J(X,-d+e)$. Taking care of the twisting, the Abel map is equivariant with respect to the sum of effective Cartier divisors: for any pair of line bundles $M$ and $N$ on $X$ and for any $D \in \Hilb_X$, $E \in \lHilb_X$, \[\A_{M\otimes N}(D+E) \simeq \A_M(D) \otimes \A_N(E). \]

\section{Review of the Fitting ideal}\label{SectionReviewFitting}
A fundamental tool for the definition of the direct image for generalized divisors is the \textit{Fitting ideal of a module} (and \textit{of a sheaf of modules}), of which we recall here the definition. See \cite[Chapter 20]{Eis},  \cite[Chapter 2.4]{Vas}  and \cite[Chapter V.1.3]{EisHar} for a detailed treatment.

\begin{deflemma}\label{FittingDef}
Let $X$ be a scheme and let $\Ff$ be a coherent sheaf on $X$. Let \[
\Ee_1 \xrightarrow{\hspace{1em}\psi\hspace{1em}} \Ee_0 \longrightarrow \Ff \rightarrow 0
\]
be any finite presentation of $\Ff$, with $\Ee_0$ locally free of rank $r$. The \textit{$0$-th Fitting ideal} of $\Ff$, denoted $\Fitt_0(\Ff)$, is defined as the image of the map: \[
\textstyle{\bigwedge^r }\Ee_1 \otimes \left(\textstyle	{\bigwedge^r } \Ee_0\right)^{-1} \xrightarrow{\hspace{1em}\Psi\hspace{1em}} \Oo_X
\]
induced by $\bigwedge^r\psi: \bigwedge^r\Ee_1 \rightarrow \bigwedge^r\Ee_0$ and is independent of the choiche of the presentation for $\Ff$. If $\Ee_1$ is locally free, then $\psi$ can be locally represented by a matrix and the $0$-th Fitting ideal is generated locally by the minors of size $r$ of such matrix, with the convention	 that the determinant of the $0 \times 0$ matrix is $1$.
\end{deflemma}
\begin{proof}
	The definition is the sheaf-theoretic reformulation of \cite[Corollary-Definition 20.4]{Eis}.
\end{proof}

As a consequence of the definition, the formation of Fitting ideals commutes with restrictions and with base change \cite[Corollary 20.5]{Eis}.

\begin{deflemma}
Let $X$ be a scheme and $\Ff$ be a coherent sheaf on $X$. The \textit{$0$-th Fitting scheme} of $\Ff$ is the subscheme of $X$ defined as the zero locus of $\Fitt_0(\Ff) \subseteq \Oo_X$ in $X$.  The $0$-th Fitting scheme contains the support of $\Ff$, as a closed subscheme with the same underlying topological space.
\end{deflemma}
\begin{proof}
The result is stated in \cite[Definition V-10]{EisHar}.
\end{proof}

\section{The direct and inverse image for generalized divisors and generalized line bundles}\label{SectionDirectImageSet}
Let $\pi: X \rightarrow Y$ be a finite, flat morphism of degree $n$ between curves. In the present section, we extend the notion of direct and inverse image from Cartier divisors to generalized divisors and generalized line bundles.

\subsection{The direct image}

Here, we define the notion of direct image for generalized divisors. First, we start with the case of an effective generalized divisor. Let $D \in \GDiv^+(X)$ be an effective generalized divisor on $X$, with ideal sheaf $\Ii$. Since $\pi$ is finite and flat, the pushforward $\pi_*( \Oo_X / \Ii )$ is a coherent $\Oo_Y$-module.

\begin{deflemma}{(Direct image of an effective generalized divisor)}
Let $D \in \GDiv^+(X)$ be an effective generalized divisor on $X$, with ideal sheaf $\Ii \subseteq \Oo_X$. The \textit{direct image of D with respect to } $\pi$, denoted with $\pi_*(D)$, is the effective generalized divisor on $Y$ defined by the $0$-th Fitting ideal of  $\pi_*( \Oo_X / \Ii )$.
\end{deflemma}

\begin{proof}
Let us prove that the 0-th Fitting ideal of  $\pi_*( \Oo_X / \Ii )$ defines an effective generalized divisor on $Y$. First note that $\Fitt_0 \pi_*( \Oo_X / \Ii )$ is a subsheaf of $\Oo_Y$, and is a coherent $\Oo_Y$-module; hence, it is an effective fractional ideal. To prove that  $\Fitt_0 \pi_*(\Oo_X / \Ii)$  is nondegenerate, consider a generic point $\eta \in Y$. Since the map $\pi$ is dominant, the preimage $\pi^{-1}(\eta) = \{ \eta_i \}$ is a finite set of generic points of $X$. Then,
\begin{align*}
\Fitt_0(\pi_*(\Oo_X / \Ii))_\eta &= \Fitt_0(\pi_*(\Oo_X / \Ii)_\eta) = \\
&=  \Fitt_0( (0)_{\eta} ) = \Oo_{X,\eta} = \Kk_{X, \eta}.
\end{align*}
\end{proof}

\begin{lemma}\label{linear} {\normalfont(Linearity w.r.t effective Cartier divisors)}
Let $D \in \GDiv^+(X)$ be an effective generalized divisor and $E \in \CDiv^+(X)$ be an effective Cartier divisor on $X$. Then, $\pi_*(E)$ is a Cartier divisor on $Y$ and $\pi_*(D+E) = \pi_*(D)+\pi_*(E)$.
\end{lemma}
\begin{proof}
Let $\Jj$ be the ideal sheaf of $E$ and let $\{V_i\}$ be an open affine cover of $Y$ such that $\pi_* \Oo_X$ and $\pi_* \Jj$ are trivial on each $V_i$ as in Lemma \ref{TrivialOnSaturated}. The sheaf $\pi_* \Oo_X/\Jj$ is locally presented by the exact sequence:
\[ (\pi_* \Oo_X)_{|V_i} \xrightarrow{\cdot h_i} (\pi_* \Oo_X)_{|V_i} \rightarrow (\pi_* \Oo_X/\Jj)_{|V_i} \rightarrow 0
\]
where $h_i \in \Gamma(\pi^{-1}(V_i), \Oo_X)$ is a local generator for $\Jj \subset \Oo_X$ on $\pi^{-1}(V_i)$. Since $(\pi_* \Oo_X)_{|V_i}$ is a free $(\Oo_Y)_{|V_i}$-module of rank $n$, there is a $n \times n$ matrix $H_i$ with entries in $(\Oo_Y)_{|V_i}$ representing the multiplication by $h_i$. Then,  by Definition \ref{FittingDef},  the 0-th Fitting ideal of $\pi_* \Oo_X/\Jj$ is the principal ideal generated locally by $\det (\cdot h_i)$ on $V_i$.  In particular, $\pi_*(E)$ is a Cartier divisor.

Let $\Ii$ be the ideal sheaf of $D$, so that $\Ii \cdot \Jj$ is the ideal sheaf of $D+E$. To prove the remaining part of the thesis, we show that the equality \[
\Fitt_0 (\pi_* \Oo_X/\Ii\Jj) = \Fitt_0 (\pi_* \Oo_X/\Ii)\cdot \Fitt_0(\pi_* \Oo_X/\Jj)
\]
holds locally around any point $y \in Y$. Let $V$ be an open neighborhood of $y$ such that $(\pi_* \Oo_X)_{|V}$ is a free $\Oo_Y$-module, $\pi_* \Ii_{|V}$ is generated by sections $s_1, \dots, s_r$ of  $\Gamma(V, \pi_* \Oo_X)$ and $\pi_* \Jj_{|V}$ is a principal ideal generated by a section $h$ of $\Gamma(V, \pi_* \Oo_X)$.  In terms of exact sequences:
\begin{align*}
(\pi_* \Oo_X)_{|V}^{\oplus r} \xrightarrow{(\cdot s_1, \dots, \cdot s_r)}  (\pi_* \Oo_X)_{|V} \rightarrow (\pi_* \Oo_X/\Ii)_{|V} \rightarrow 0 \\
(\pi_* \Oo_X)_{|V} \xrightarrow{\cdot h} (\pi_* \Oo_X)_{|V} \rightarrow (\pi_* \Oo_X/\Jj)_{|V} \rightarrow 0
\end{align*}
The ideal sheaf $\Ii \cdot \Jj$ is equal to $\Jj \cdot \Ii$, which is generated on $V$ by the sections $h s_1 , \dots, h s_n $. In terms of exact sequences:
\begin{align*}
(\pi_* \Oo_X)_{|V}^{\oplus r} \xrightarrow{( \cdot h s_1, \dots, \cdot h s_r )}  (\pi_* \Oo_X)_{|V} \rightarrow (\pi_* \Oo_X/\Ii\Jj)_{|V} \rightarrow 0 
\end{align*}
Denote with $S_i$ and $H$ the $(\Oo_Y)_{|V}$-matrices representing the multiplication by $s_i$ and $h$ respectively.  The map $( \cdot h s_1, \dots,  \cdot h s_r )$ in the previous exact sequence is represented by the $n \times nr$ matrix
\[ M= 
\left[
\begin{array}{c|c|c}
H S_1 & \dots & H S_r
\end{array}
\right]
 = H \left[
 \begin{array}{c|c|c}
  S_1 & \dots & S_r
 \end{array}
 \right]
\]
The 0-th Fitting ideal of $\pi_* \Oo_X/\Ii\Jj$, restricted to $V$, is the ideal of $(\pi_* \Oo_X)_{|V}$ generated by the $n \times n$ minors of the matrix $M$. Any such minor is equal to a $n \times n$  minor of the matrix $\left[
\begin{array}{c|c|c}
S_1 & \dots & S_r
\end{array}
\right]$ multiplied by $\det H$. Then, by Definition \ref{FittingDef}, 
\[
\Fitt_0 (\pi_* \Oo_X/\Ii\Jj)_{|V} = \Fitt_0 (\pi_* \Oo_X/\Ii)_{|V} \cdot \Fitt_0(\pi_* \Oo_X/\Jj)_{|V}
\]
\end{proof}

\begin{deflemma}\label{DirectImageDef} {(Direct image of a generalized divisor)}
Let $D \in \GDiv(X)$ be a generalized divisor on $X$, such that $D=D'-E$ with  $D' \in \GDiv^+(X)$ and $E \in \CDiv^+(X)$ by Lemma \ref{DifferenceOfEffective}. The \textit{direct image of D with respect to $\pi$}, denoted with $\pi_*(D)$, is the generalized divisor $\pi_*(D') - \pi_*(E)$.
\end{deflemma}
\begin{proof}
To prove that it is well defined, let $D=D'-E=\widetilde{D}'-\widetilde{E}$, with $D', \widetilde{D}'$ effective and $E, \widetilde{E}$ effective Cartier. 
Since $E, \widetilde{E}$ are Cartier, $D'+ \widetilde{E}=\widetilde{D}' +E$; then, by Lemma \ref{linear} we have:
	\[\pi_*(D') + \pi_*(\widetilde{E}) = \pi_*(\widetilde{D}') + \pi_*(E).\]
Since $\pi_*(E)$ and $\pi_*(\widetilde{E})$ are also Cartier by Lemma \ref{linear}, they can be subtracted from each side in order to obtain:
\[ \pi_*(D') - \pi_*(E) = \pi_*(\widetilde{D}') - \pi_*(\widetilde{E}). \]
This shows that $\pi_*(D)$ does not depend on the choice of $D'$ and $E$.
\end{proof}

We study now some properties of the direct image for generalized divisors.

\begin{prop}{\normalfont (Properties of direct image) }\label{DirectImageProperties}
\begin{enumerate}
	\item Let $D \in \CDiv(X)$ be a Cartier divisor. Then, $\pi_*(D)$ is a Cartier divisor and $\pi_*(-D)=-\pi_*(D)$. Moreover, $\pi_*(D)$ coincides with $\pi_*(D)$ of Definition \ref{DirectInverseImageForCartDef}.
	\item Let $D,E \in \GDiv(X)$ be generalized divisors, such that $E$ is Cartier. Then, $\pi_*(D+E) = \pi_*(D) + \pi_*(E)$.
	\item Let $V \subset Y$ be an open subset, and denote with $\pi_U$ the restriction of $\pi$ to $U = \pi^{-1}(V) \subset X$. Let $D \in \GDiv(X)$ be a generalized divisor on $X$. Then,
\[ (\pi_{U})_*(D_{|U}) = \pi_*(D)_{|V} \]
	\item Let $D, D' \in \GDiv(X)$ be generalized divisors such that $D \sim D'$. Then, $\pi_*(D) \sim \pi_*(D')$.
\end{enumerate}
\end{prop}
\begin{proof}
To prove (1), consider $D=E-F$ with $E,F$ effective and $F$ Cartier by Lemma \ref{DifferenceOfEffective}. Since $D$ is Cartier and $E=D+F$, then also $E$ is Cartier. By Definition \ref{DirectImageDef}, \[
\pi_*(D)=\pi_*(E)-\pi_*(F).
\]
By Lemma \ref{linear}, it is a difference of Cartier divisors and hence it is Cartier. To compute $\pi_*(-D)$, note that since $F$ is Cartier then $-D=F-E$, and it is a difference of effective Cartier divisors; then, apply Definition \ref{DirectImageDef} to obtain \begin{align*}
\pi_*(-D) &= \pi_*(F)-\pi_*(E) = \\
&= -\pi_*(D).
\end{align*}
To compare $\pi_*(D)$ with Definition \ref{DirectInverseImageForCartDef}, let $\Ii$ and $\Jj$ be the ideal sheaves of $E$ and $F$ and let $\{V_i\}_{i \in I}$ be an open cover of $Y$ such that $\pi_*\Ii_{|V_i}$  and $\pi_*\Jj_{|V_i}$ are non-degenerate principal ideals of $(\pi_*\Oo_X)_{|V_i}$-modules generated by the regular sections $f_i$ and $g_i$ of $\Gamma(V, \pi_*\Oo_X)$, respectively on each $i \in I$. The fractional ideal of $D$ is then generated on each $V_i$ by the meromorphic regular section $h_i = f_i/g_i$ of $\Gamma(V_i, \pi_*\Kk_X^*)$.
By the proof of Lemma \ref{linear}, the ideal sheaves of $\pi_*(E)$ and $\pi_*(F)$ are generated on each $V_i$ by $\det(\cdot f_i)$ and $\det(\cdot g_i)$ , so the fractional ideal of $\pi_*(D)$ is generated on each $V_i$ by the meromorphic regular section $\det(\cdot f_i) / \det(\cdot g_i)$ of $\Gamma(V, \Kk_Y^*)$. Using the same cover in Definition \ref{DirectInverseImageForCartDef} and applying Definition \ref{NormOfSheaves}, the sheaf $\pi_*(D)$ defined above and the sheaf $\pi_*(D)$ defined in Section \ref{SectionReviewNorm} have the same local generators, hence they are equal.

To prove (2), consider $D=D_1-D_2$ and $E=E_1-E_2$, with $D_1, D_2, E_1, E_2$ effective and $D_2, E_2$ Cartier. Note that $D+E=(D_1+E_1)-(D_2+E_2)$, and it is a difference of effective divisors with $D_2+E_2$ Cartier. Then, applying Definition \ref{DirectImageDef} and Lemma \ref{linear}, we obtain:
\begin{align*}
	\pi_*(D+E) &= \pi_*((D_1 +E_1)-(D_2+E_2)) = \\
	&=\pi_*(D_1+E_1)-\pi_*(D_2+E_2)=\\
	&=\pi_*(D_1)+\pi_*(E_1)-\pi_*(D_2)-\pi_*(E_2)=\\
	&=(\pi_*(D_1)-\pi_*(D_2))-(\pi_*(E_1)-\pi_*(E_2) ) =\\
	&=\pi_*(D)+\pi_*(E).
\end{align*}

To prove (3), if $D$ is effective, the result follows from Definition \ref{FittingDef}. Then, observe that the operations of product and inverse of fractional ideals commute with restrictions.

To prove (4), let $f \in \Gamma(X, \Kk_X)$ be a global section that generates a principal divisor $E=(f) \in \Prin(X)$. By linearity, it is sufficient to prove that $\pi_*(E)$ is again principal. Let $\{ V_i \}$ be an affine open cover of $Y$, such that locally \[
f_{|\pi^{-1}(V_i)}  = g_i / h_i, \hspace{3em} g_i, h_i \in \Gamma(\pi^{-1}(V_i), \Oo_X)\]
In terms of divisors, this means that $(f)_{|\pi^{-1}(V_i)}=(g_i)-(h_i)$, and it is a difference of effective Cartier divisors on $\pi^{-1}(V_i)$. Then, applying Part (3) and reasoning simiarly to Part (1), we obtain: \begin{align*}
\pi_*( (f) )_{|V} &= (\pi_i)_*( (g_i) ) - (\pi_i)_*( (h_i) ) =  \\
&= 	( \det [\cdot g_i] ) - ( \det [\cdot h_i] )  = \\
&= 	( \det [\cdot g_i] / \det [\cdot h_i] ).
\end{align*}
Since taking determinants is multiplicative, the local sections $\det [\cdot g_i] / \det [\cdot h_i]$ glue together to give a global section $\widetilde f$ of $\Kk_Y$, such that $\pi_*( E ) = ( \widetilde{f} )$.
\end{proof}

We are now ready to define the direct image for generalized line bundles. 

\begin{deflemma}{(Direct image for generalized line bundles)}\label{DirectImageForGPic}
The direct image for generalized divisors induce a direct image map between the sets of generalized line bundles, defined as:
\begin{align*}
[\pi_*]: \hspace{2em} \GPic(X) &\longrightarrow \GPic(Y) \\
[D] &\mapsto [\pi_*(D)].
\end{align*}
\end{deflemma} 
\begin{proof}
Recall that for any curve $X$, the set $\GPic(X)$ can be seen equivalently as the set of generalized line bundles or as the set of generalized divisors modulo linear equivalence. Here, $[\pi_*]$ is defined in terms of generalized divisors modulo linear equivalence. By Proposition \ref{DirectImageProperties}, the direct images of linearly equivalent divisors are linearly equivalent, hence  $[\pi_*]$ is well defined.
\end{proof}

In the remaining part of this subsection, we study an alternative formula for $[\pi_*]$ in terms of generalized line bundles. First, we need a technical lemma.

\begin{lemma}\label{TorsionVsDoubleDual}
Let $\Ff$ be a coherent sheaf on a curve $X$ which is locally free of rank 1 at any generic point and let $\omega=\omega_X$ be the canonical, or dualizing sheaf of $X$. Let $T(\Ff)$ be the torsion subsheaf of $\Ff$ and $\Ff^{\omega\omega}=\Hhom(\Hhom(\Ff, \omega), \omega)$ be the double $\omega$-dual. There, there is a canonical isomorphism 
\[
\Ff^{tf}= \Ff/T(\Ff) \xrightarrow{\hspace{1em}\sim\hspace{1em}} \Ff^{\omega\omega}. 
\]
\end{lemma}
\begin{proof}
The sheaf $\Ff^{tf}$ is endowed with a quotient map $q: \Ff \twoheadrightarrow \Ff^{tf}$, which is universal among all arrows from $\Ff$ to torsion-free sheaves (i.e. pure of dimension 1). Note that, since $\Ff$ is locally free of rank 1 at any generic point of $X$, then $q$ is generically an isomorphism. Let $\alpha(\Ff): \Ff \rightarrow \Ff^{\omega\omega}$ be the canonical map from $\Ff$ to its double $\omega$-dual. Since taking double $\omega$-duals is functorial, there is a commutative diagram of homomorphism of sheaves:
\[
\begin{tikzcd}
\Ff \arrow{r}{q} \arrow{d}{\alpha(\Ff)}  & \Ff^{tf} \arrow{d}{\alpha(\Ff^{tf})} \\
\Ff^{\omega\omega} \arrow{r}{q^{\omega\omega}} & (\Ff^{tf})^{\omega\omega}.
\end{tikzcd}
\]
Using \cite[Proposition 1.5]{Har07} and recalling that $S_1 = S_2$ for sheaves on curves, we note that a sheaf on $X$ is $\omega$-reflexive if and only if it is 
$S_1$. In particular, $\Ff^{\omega\omega}$ is $\omega$-reflexive by  \cite[Proposition 1.6]{Har07} and hence it is $S_1$. Then, by the universal property of $q$, there is a unique map $\psi: \Ff^{tf} \rightarrow \Ff^{\omega\omega}$ such that $\psi \circ q = \alpha(\Ff)$. Note that \[
q^{\omega\omega} \circ \psi \circ q = q^{\omega\omega}  \circ \alpha(\Ff) = \alpha(\Ff^{tf}) \circ q,\] so by surjectivity of $q$ we conclude that $q^{\omega\omega} \circ \psi = \alpha(\Ff^{tf})$.

We show moreover that $\psi$ is an isomorphism. By construction $\Ff^{tf}$ is pure, hence it is $\omega$-reflexive and $\alpha(\Ff^{tf})$ is an isomorphism. Moreover, $q^{\omega\omega}$  is surjective and generically an isomorphism since $q$ is surjective and generically an isomorphism.  Then, the kernel $K$ of $q^{\omega\omega}$ is a subsheaf of $\Ff^{\omega\omega}$ which is generically zero. Since $\Ff^{\omega\omega}$ is pure, we conclude that $K$ is everywhere zero and then $q^{\omega\omega}$ is an isomorphism. This shows that $\psi=(q^{\omega\omega})^{-1} \circ \alpha(\Ff^{tf})$ is an isomorphism.
\end{proof}

In order to study the formula for $[\pi_*]$, we study a preliminary formula for $\pi_*$ in the case of effective generalized divisors.

\begin{lemma}\label{FormulaForEffective}
Let $D \in \GDiv^+(X)$ be an effective generalized divisor with ideal sheaf $\Ii \subset \Oo_X$. Then, there is an injection: \[
\left(\textstyle{\bigwedge^n}\big( \pi_* \Ii \big)\right)^{\omega\omega} \otimes \det( \pi_* \Oo_X)^{-1} \hookrightarrow \det( \pi_* \Oo_X) \otimes \det( \pi_* \Oo_X)^{-1} \xrightarrow{{}_\sim} \Oo_Y
\]
whose image in $\Oo_Y$ is the $0$-th Fitting ideal of  $\pi_*\Oo_X/\Ii$.
\end{lemma}
\begin{proof}
Consider the short exact sequence: \[
0 \rightarrow \Ii \rightarrow \Oo_X \rightarrow \Oo_D \rightarrow 0.
\]
Since $\pi$ is finite and flat, the pushforward induces a a short exact sequence:\[
0 \rightarrow \pi_*\Ii \xrightarrow{\varphi} \pi_*\Oo_X \rightarrow \pi_*(\Oo_X/\Ii) \rightarrow 0.
\]
In particular, the last exact sequence is a finite presentation of $\pi_*(\Oo_X/\Ii)$ whose middle term is locally free. Hence, by Definition \ref{FittingDef}, the $0$-th Fitting ideal of $\pi_*(\Oo_X/\Ii)$ is equal to the image of the morphism: \[
\textstyle{\bigwedge^n}\big( \pi_* \Ii \big) \otimes \det \big( \pi_* \Oo_X \big)^{-1} \xrightarrow{\hspace{.5em}\det\varphi \otimes 1 \hspace{.5em}} \det( \pi_* \Oo_X) \otimes \det( \pi_* \Oo_X)^{-1} \xrightarrow{{}_\sim} \Oo_Y.
\]
Consider now the determinant map $\det\varphi$. Since $\det(\pi_*\Oo_X)$ is also a pure sheaf, applying the universal property of the torsion-free quotient together with Lemma \ref{TorsionVsDoubleDual} we obtain the following commutative diagram:
\[
\begin{tikzcd}
\bigwedge^n\big( \pi_* \Ii \big)  \arrow{r}{\det\varphi} \arrow{d}{\alpha}  & \det(\pi_*\Oo_X)  \\
\left(\bigwedge^n\big( \pi_* \Ii \big)\right)^{\omega\omega} \arrow{ru}[swap]{\beta} & {}
\end{tikzcd}
\]
The canonical map $\alpha$ is surjective by Lemma \ref{TorsionVsDoubleDual}. Since $\Ii$ is locally free of rank 1 at the generic points of $X$, both $\det\phi$ and $\alpha$ are generically isomorphisms; hence, the map $\beta$ is generically an isomorphism. Since its domain $\left(\bigwedge^n\big( \pi_* \Ii \big)\right)^{\omega\omega} $ is pure, we conclude that its kernel is zero, hence $\beta$ is injective. We have then factorized $\det\varphi$ as the composition of a surjective map $\alpha$ followed by an injective map $\beta$. 

Tensoring by $\det (\pi_* \Oo_X)^{-1}$, we obtain then the following commutative diagram:
\[
\begin{tikzcd}
\bigwedge^n\big( \pi_* \Ii \big) \otimes \det (\pi_* \Oo_X)^{-1} \arrow{r}{\det\varphi \otimes 1} \arrow{d}{\alpha \otimes 1}  & \det(\pi_*\Oo_X) \otimes \det (\pi_* \Oo_X)^{-1} \arrow{r}{\sim}[swap]{\eta} & \Oo_Y \\
\left(\bigwedge^n\big( \pi_* \Ii \big)\right)^{\omega\omega} \otimes \det (\pi_* \Oo_X)^{-1} \arrow{ru}[swap]{\beta \otimes 1} & {} & {}
\end{tikzcd}
\]
The map $\alpha \otimes 1$ is surjective, so the composition $\eta \circ (\beta \otimes 1)$ is an injective map whose image in $\Oo_Y$ is equal to the image of the map $\eta \circ (\det\varphi \otimes 1)$. By the previous remark, such image coincides with the $0$-th Fitting ideal of $\pi_*( \Oo_X / \Ii )$, proving the lemma.
\end{proof}

We are now ready to give a sheaf-theoretic formula for $[\pi_*]$. Recall that, for any curve $X$, the set $\GPic(X)$ can be seen as the set of isomorphism classes of generalized line bundles on on $X$.

\begin{prop}{\normalfont (Formula for the direct image of generalized line bundles) } \label{GeneralizedDetFormula}
Let $\Ll$ be a generalized line bundle on $X$. Then, \[
[\pi_*](\Ll) \simeq  \left(\textstyle{\bigwedge^n}\big( \pi_* \Ll \big)\right)^{\omega\omega} \otimes \det( \pi_* \Oo_X)^{-1}  .
\]
\end{prop}
\begin{proof}
By the surjectivity of $\GDiv(X) \twoheadrightarrow \GPic(X)$, we can pick a generalized divisor $D \in \GDiv(X)$ with fractional ideal $\Ii$ isomorphic to $\Ll$; then, \[
[\pi_*](\Ll) = [\pi_*]([D]) = [\pi_*(D)] 
\] by Definition \ref{DirectImageForGPic}. By Lemma \ref{DifferenceOfEffective}, there are effective generalized divisors $E, F$ on $X$ such that $D=E-F$ and $F$ Cartier. Denote with $\Ii'$ the ideal sheaf of $E$ and with $\Jj$ the ideal sheaf of $F$. Since $F$ is Cartier, the condition $D=E-F$ can be rewritten as $E=D+F$, or in terms of sheaves: \[ \Ii' = \Ii \cdot \Jj . \]
Consider the direct images of $E$ and $F$. By Lemma \ref{FormulaForEffective}, the ideal sheaf of $\pi_*(E)$ is isomorphic to \[\left(\textstyle{\bigwedge^n}\big( \pi_* \Ii' \big)\right)^{\omega\omega} \otimes \det( \pi_* \Oo_X)^{-1},\] while the ideal sheaf of $\pi_*(F)$ is isomorphic to \[\left(\textstyle{\bigwedge^n}\big( \pi_* \Jj \big)\right)^{\omega\omega} \otimes \det( \pi_* \Oo_X)^{-1} \simeq \textstyle{\bigwedge^n}\big( \pi_* \Jj \big)\otimes \det( \pi_* \Oo_X)^{-1}. \]
By Definition \ref{DirectImageDef}, $\pi_*(D)=\pi_*(E)-\pi_*(F)$. Then, the fractional ideal of $\pi_*(D)$ is isomorphic to: \begin{align*}
\left(\textstyle{\bigwedge^n}\big( \pi_* \Ii' \big)\right)^{\omega\omega} \otimes \det( \pi_* \Oo_X)^{-1} \otimes \left(\textstyle{\bigwedge^n}\big( \pi_* \Jj \big)\right)^{-1} \otimes \det( \pi_* \Oo_X),
\end{align*}
which in turn is isomorphic to:  \[ \left(\textstyle{\bigwedge^n}\big( \pi_* \Ii' \big)\right)^{\omega\omega}  \otimes \left(\textstyle{\bigwedge^n}\big( \pi_* \Jj \big)\right)^{-1}.\]
Then, we are left to prove that: \[
\left(\textstyle{\bigwedge^n}\big( \pi_* \Ii' \big)\right)^{\omega\omega}  \otimes 
\left(\textstyle{\bigwedge^n}\big( \pi_* \Jj \big)\right)^{-1} \simeq
\left(\textstyle{\bigwedge^n}\big( \pi_* \Ii \big)\right)^{\omega\omega} \otimes \det( \pi_* \Oo_X)^{-1},
\]
or, equivalently, that: \[
\left(\textstyle{\bigwedge^n}\big( \pi_* \Ii' \big)\right)^{\omega\omega}  \otimes  \det( \pi_* \Oo_X) \simeq
\left(\textstyle{\bigwedge^n}\big( \pi_* \Ii \big)\right)^{\omega\omega} \otimes 
\left(\textstyle{\bigwedge^n}\big( \pi_* \Jj \big)\right),
\]
under the assumptions $\Ii' = \Ii \cdot \Jj$ and $\Jj$ locally principal. Consider an open cover $\{V_i\}_{i \in I}$ of $Y$ such that $\Jj$ is trivial on each $U_i = \pi^{-1}(V_i)$, i.e. there is an isomorphism $\lambda_i: \Jj_{|U_i} = (f_i) \xrightarrow{{}_\sim} \Oo_{X|U_i}$ 
for each $i \in I$. On the intersections $U_i \cap U_j$, the collection $\{\lambda_i \circ \lambda_j^{-1}\}$ of automorphism of  ${\Oo_X}_{|U_i \cap U_j}$ is a cochain of sections of ${\Oo_X^*}_{|U_i \cap U_j}$ that measures the obstruction for the $\lambda_i$'s to glue to a global isomorphism. We define now an isomorphism \[
\alpha: \left(\textstyle{\bigwedge^n}\big( \pi_* \Ii' \big)\right)^{\omega\omega}  \otimes  \det( \pi_* \Oo_X) \longrightarrow
\left(\textstyle{\bigwedge^n}\big( \pi_* \Ii \big)\right)^{\omega\omega} \otimes 
\left(\textstyle{\bigwedge^n}\big( \pi_* \Jj \big)\right)
\] by glueing a collection of isomorphisms $\alpha_i$ defined on each $V_i$. To do so, we define each $\alpha_i$ as the following composition of arrows:
\vspace{1em}

\begin{tikzcd}
\left(\left(\bigwedge^n\big( \pi_* \Ii' \big)\right)^{\omega\omega} \otimes \det(\pi_*\Oo_X)\right)_{|V_i} \arrow{r}{\alpha_i} \arrow{d}[phantom, anchor=center,rotate=-90,yshift=-1ex]{=}  &  \left( \left( \bigwedge^n \big( \pi_* \Ii \big) \right)^{\omega\omega} \otimes \bigwedge^n\big( \pi_* \Jj \big)\right)_{|V_i} \\
\left(\bigwedge^n\big( \pi_* \Ii'_{|U_i} \big)\right)^{\omega\omega} \otimes \det(\pi_*\Oo_{X|U_i}) \arrow{d}[swap]{\left(\bigwedge^n (\pi_*\lambda_i)\right)^{\omega\omega} \otimes \id }  & \left(\bigwedge^n \big( \pi_* \Ii_{|U_i} \big) \right)^{\omega\omega} \otimes \left( \bigwedge^n\big( \pi_* \Jj_{|U_i} \big) \right)^{\omega\omega} \arrow{u}[anchor=center,rotate=-90,yshift=1ex]{\simeq}  \\
\left(\bigwedge^n\big( \pi_* \Ii_{|U_i} \big)\right)^{\omega\omega} \otimes \det(\pi_*\Oo_{X|U_i}) \arrow[r, phantom, "\simeq"] &
 \left(\bigwedge^n\big( \pi_* \Ii_{|U_i} \big)\right)^{\omega\omega} \otimes \left(\bigwedge^n\big(\pi_*\Oo_{X|U_i} \big)\right)^{\omega\omega}
 \arrow{u}[swap]{\id \otimes \left(\bigwedge^n (\pi_* \lambda_i^{-1})\right)^{\omega\omega}} 
\end{tikzcd}
\vspace{1em}

\noindent i.e. $\alpha_i := \left(\bigwedge^n\big(\pi_*\lambda_i\big)\right)^{\omega\omega} \otimes \left( \bigwedge^n \big(\pi_* \lambda_i^{-1}\big) \right)^{\omega\omega}$. Since $\bigwedge^n\big(\pi_*\_ \big)$ and $(\_)^{\omega\omega}$ are functorial, the obstruction $\alpha_i \circ \alpha_j^{-1}$ is trivial on any $V_i \cap V_j$, whence the $\alpha_i$'s glue together to a global isomorphism $\alpha$.
\end{proof}
\begin{cor}\label{NmVsPiStar} Let $\Ll$ be a line bundle on $X$. Then, \[
\Nm_\pi(\Ll) = [\pi_*](\Ll).
\]
\end{cor}
\begin{proof}
By Lemma \ref{DetFormula}, \[
\Nm_\pi(\Ll) \simeq \det(\pi_* \Ll) \otimes \det(\pi_* \Oo_X)^{-1}.
\]
On the other side, by Proposition \ref{GeneralizedDetFormula}, \[
[\pi_*](\Ll) \simeq \left( \det(\pi_* \Ll)\right)^{\omega\omega} \otimes \det(\pi_* \Oo_X)^{-1}.
\]
Since $\Ll$ is locally free, $\det(\pi_* \Ll)$ is a line bundle and in particular is a pure coherent sheaf. Then, $\det(\pi_* \Ll) \simeq \left(\det(\pi_* \Ll)\right)^{\omega\omega}$, proving the thesis.
\end{proof}

\begin{cor}
	Let $\Ll$ be a generalized line bundle on $X$. Suppose that $Y$ is smooth.  Then: \[
	[\pi_*](\Ll) \simeq \textstyle{\bigwedge^n}\big( \pi_* \Ll \big) \otimes \det( \pi_* \Oo_X)^{-1}.
	\]
\end{cor}
\begin{proof}
	First note that, by Proposition \ref{GeneralizedDetFormula}, \[
	[\pi_*](\Ll) \simeq \left(\textstyle{\bigwedge^n}\big( \pi_* \Ll \big)\right)^{\omega\omega} \otimes \det( \pi_* \Oo_X)^{-1}  . \]
	Second, observe that the pure sheaf $\pi_* \Ll$ is locally free of rank $n$ since $Y$ is smooth (see \cite[Example 1.1.16]{Huy}), and $\textstyle{\bigwedge^n}\big( \pi_* \Ll \big)$ is a line bundle. Then, \[
	\left(\textstyle{\bigwedge^n}\big( \pi_* \Ll \big)\right)^{\omega\omega} \simeq \textstyle{\bigwedge^n}\big( \pi_* \Ll \big).
	\]
	Combining  these two facts, we have proved the thesis.
\end{proof}

\subsection{The inverse image}
In this subsection, we define the inverse image for generalized divisors and generalized line bundles and we study the relation of the inverse image with the direct image. We start from the case of effective divisors.

\begin{deflemma} {(Inverse image of an effective generalized divisor)}\label{InverseImageEffDef}
	Let $D \in \GDiv^+(Y)$ be an effective generalized divisor with ideal sheaf $\Ii \subseteq \Oo_Y$. The \textit{inverse image of $D$ relative to $\pi$}, denoted $\pi^*(D)$, is the effective generalized divisor with ideal sheaf $\pi^{-1}(\Ii) \cdot \Oo_X$.
\end{deflemma}
\begin{proof}
The	inverse image ideal $\pi^{-1}(\Ii)$ is an ideal sheaf of $\pi^{-1}(\Oo_Y)$-modules and can be extended to $\Oo_X$-modules via the canonical map of sheaves of rings $\pi^\sharp: \pi^{-1}(\Oo_Y) \rightarrow \Oo_X$. It is coherent since $\Ii$ is coherent. If $\eta$ is a generic point of $X$, then $\pi(\eta)$ is a generic point of $Y$, hence \[
(\pi^{-1}(\Ii)\cdot \Oo_X)_\eta =  \Ii_{\pi(\eta)} \cdot \Oo_{X,\eta} = \Oo_{X, \eta}
\] since $\Ii_{\pi(\eta)} = \Oo_{Y, \pi(\eta)}$.
\end{proof}

\begin{rmk}\label{InverseImageVsPullback}
In the setting of Definition \ref{InverseImageEffDef}, consider the short exact sequence:
	\[ 0 \rightarrow \Ii \rightarrow \Oo_Y \rightarrow \Oo_D \rightarrow 0. \]
	Since $\pi$ is flat and surjective, the pullback functor $\_ \otimes_{\pi^{-1}\Oo_Y}\Oo_X$ is exact as well as the inverse image functor $\pi^{-1}$.  Then, the previous exact sequence induces the following exact sequence:
	\[ 0 \rightarrow \pi^* \Ii \rightarrow \pi^*\Oo_Y \rightarrow \pi^*\Oo_D \rightarrow 0. \]
	Since $\pi^* \Oo_Y = \Oo_X$, the pullback sheaf $\pi^*(\Ii)$ has a canonical injection in $\Oo_X$, whose image is exactly the inverse image ideal $\pi^{-1}(\Ii) \cdot \Oo_X$ 
\end{rmk}

\begin{lemma}\label{InverseLinear}
	Let $D \in \GDiv^+(Y)$ be a generalized effective divisor and $E \in \CDiv^+(Y)$ be an effective Cartier divisor on $Y$. Then, the inverse image divisor $\pi^*(E)$ is Cartier and $\pi^*(D+E) = \pi^*(D)+\pi^*(E)$.
\end{lemma}
\begin{proof}
Let $\Ii$ and $\Jj$ be the ideal sheaves of $D$ and $E$ respectively. The ideal sheaf $\Jj$ is locally principal, hence  its inverse image $\pi^{-1}(\Jj) \cdot \Oo_X$ is again locally principal; then, $\pi^*(E)$ is Cartier. The generalized divisor $D+E$ is defined by the ideal sheaf $\Ii \cdot \Jj$, whose inverse image is: \begin{align*}
\pi^{-1}(\Ii \cdot \Jj) \cdot \Oo_X &= (\pi^{-1}(\Ii) \cdot \pi^{-1}(\Jj)) \cdot \Oo_X = \\ 
&=(\pi^{-1}(\Ii) \cdot \Oo_X) \cdot (\pi^{-1}(\Jj) \cdot \Oo_X),
\end{align*}
which is the defining ideal of $\pi^*(D) + \pi^*(E)$.
\end{proof}

Now, we can extend the definition of inverse image to any generalized divisor and study its properties.

\begin{deflemma} {(Inverse image of a generalized divisor)}\label{InverseImageDef}
Let $D \in \GDiv(Y)$ be any generalized divisor and let $D=E-F$ with $E,F$ effective generalized divisors and $F$ Cartier by Lemma \ref{DifferenceOfEffective}. The \textit{inverse image of $D$ relative to $\pi$}, denoted $\pi^*(D)$, is the generalized divisor $\pi^*(E)-\pi^*(F)$.
\end{deflemma}
\begin{proof}
To prove that it is well defined, let $D=D'-E=\widetilde{D}'-\widetilde{E}$, with $D', \widetilde{D}'$ effective and $E, \widetilde{E}$ effective Cartier. 
Since $E, \widetilde{E}$ are Cartier, $D'+ \widetilde{E}=\widetilde{D}' +E$; then, by Lemma \ref{InverseLinear} we have:
	\[\pi^*(D') + \pi^*(\widetilde{E}) = \pi^*(\widetilde{D}') + \pi^*(E).\]
Again by Lemma \ref{InverseLinear} $\pi^*(E)$ and $\pi^*(\widetilde{E})$ are Cartier, so they can be subtracted from each side in order to obtain:
\[ \pi^*(D') - \pi^*(E) = \pi^*(\widetilde{D}') - \pi^*(\widetilde{E}), \]
so $\pi^*(D)$ does not depend on the choice of $D'$ and $E$.
\end{proof}

\begin{prop} {\normalfont (Properties of inverse image) }\label{InverseImageProperties}
\begin{enumerate}
	\item Let $D \in \CDiv(Y)$ be a Cartier divisor on $Y$. Then, $\pi^*(D)$ is a Cartier divisor and $\pi^*(-D)=-\pi^*(D)$. Moreover, $\pi^*(D)$ coincides with $\pi^*(D)$ of Definition \ref{DirectInverseImageForCartDef}.
	\item Let $D,E \in \GDiv(Y)$ be generalized divisors, such that $E$ is Cartier. Then, $\pi^*(D+E) = \pi^*(D) + \pi^*(E)$.
	\item Let $V \subset Y$ be an open subset, and denote with $\pi_U$ the restriction of $\pi$ to $U = \pi^{-1}(V) \subset X$. Let $D \in \GDiv(Y)$ be a generalized divisor on $Y$. Then,
\[ (\pi_{U})^*(D_{|V}) = \pi^*(D)_{|U} \]
	\item Let $D, D' \in \GDiv(Y)$ be generalized divisors such that $D \sim D'$. Then, $\pi^*(D) \sim \pi^*(D')$.
\end{enumerate}
\end{prop}
\begin{proof}
	To prove (1), consider $D=E-F$ with $E,F$ effective and $F$ Cartier by Lemma \ref{DifferenceOfEffective} on $Y$. Since $D$ is Cartier and $E=D+F$, then also $E$ is Cartier. By Definition \ref{InverseImageDef}, \[
	\pi^*(D)=\pi^*(E)-\pi^*(F).
	\]
	By Lemma \ref{InverseLinear}, it is a difference of Cartier divisors and hence it is Cartier. To compute $\pi(-D)$, note that since $F$ is Cartier then $-D=F-E$, and it is a difference of effective Cartier divisors; then, apply Definition \ref{InverseImageDef} to obtain \begin{align*}
	\pi^*(-D) &= \pi^*(F)-\pi^*(E) = \\
	&= -\pi^*(D).
	\end{align*}
	To compare $\pi^*(D)$ with Definition \ref{DirectInverseImageForCartDef}, let $\Ii$ and $\Jj$ be the ideal sheaves of $E$ and $F$ and let $\{V_i\}_{i \in I}$ be an open cover of $Y$ such that $\Ii_{|V_i} and $ $\Jj_{|V_i}$ are principal ideals of $\Oo_{Y|V_i}$-modules generated by regular sections $s_i$ and $t_i$ of $\Gamma(V_i, \Oo_Y)$, respectively on each $i \in I$. The fractional ideal of $D$ is generated on each $V_i$ by the meromorphic regular section $u_i=s_i/t_i$ of $\Gamma(V_i, \Kk_Y)$. By Definition \ref{InverseImageEffDef}, the ideal sheaves of $\pi^*(E)$ and $\pi^*(F)$ are generated on each $U_i = \pi^{-1}(V_i)$ by $\pi_{U_i}^\sharp(s_i)$ and $\pi_{U_i}^\sharp(t_i)$ respectively. Then, by Definition \ref{InverseImageDef}, the fractional ideal of $\pi^*(D)$ is generated on each $U_i$ by the meromorphic regular section $\pi_{U_i}^\sharp(s_i)/\pi_{U_i}^\sharp(t_i)$ of $\Gamma(U_i, \Kk_X)$. These are exactly the local generators for $\pi^*(D)$ as defined in Definition \ref{DirectInverseImageForCartDef}.

	To prove (2), consider $D=D_1-D_2$ and $E=E_1-E_2$, with $D_1, D_2, E_1, E_2$ effective and $D_2, E_2$ Cartier. Note that $D+E=(D_1+E_1)-(D_2+E_2)$, and it is a difference of effective divisors with $D_2+E_2$ Cartier. Then, applying Definition \ref{InverseImageDef} and Lemma \ref{InverseLinear}, we obtain:
	\begin{align*}
	\pi^*(D+E) &= \pi^*((D_1 +E_1)-(D_2+E_2)) = \\
	&=\pi^*(D_1+E_1)-\pi^*(D_2+E_2)=\\
	&=\pi^*(D_1)+\pi^*(E_1)-\pi^*(D_2)-\pi^*(E_2)=\\
	&=(\pi^*(D_1)-\pi^*(D_2))-(\pi^*(E_1)-\pi^*(E_2) ) =\\
	&=\pi^*(D)+\pi^*(E).
	\end{align*}
	
	To prove (3), if $D$ is effective, the result follows the fact that the inverse image functor $\pi^{-1}$ commutes with restrictions. Then, observe that the operations of product and inverse of fractional ideals also commute with restrictions.
	
	To prove (4), let $f \in \Gamma(Y, \Kk_Y)$ be a global section that generates a principal divisor $E=(f) \in \Prin(Y)$ such that $D=D'+E$. The inverse image $\pi^*(E)$ is the principal divisor $(\pi^\sharp(f)) \in \Prin(X)$. Then, by Lemma \ref{InverseLinear}, $\pi^*(D)=\pi^*(D')+(\pi^\sharp(f))$; so $\pi^*(D)$ and $\pi^*(D')$ are linearly equivalent.
\end{proof}

We are now ready to define the inverse image for generalized line bundles.

\begin{deflemma}{(Inverse image for generalized line bundles)}\label{InverseImageForGPic}
	The inverse image for generalized divisors induce a inverse image map between the sets of generalized line bundles, defined as:
	\begin{align*}
	[\pi^*]: \hspace{2em} \GPic(Y) &\longrightarrow \GPic(X) \\
	[D] &\mapsto [\pi^*(D)].
	\end{align*}
\end{deflemma} 
\begin{proof}
	Recall that for any curve $X$, the set $\GPic(X)$ can be seen equivalently as the set of generalized line bundles or as the set of generalized divisors modulo linear equivalence. Here, $[\pi^*]$ is defined in terms of generalized divisors modulo linear equivalence. By Proposition \ref{InverseImageProperties}, the direct images of linearly equivalent divisors are linearly equivalent, hence  $[\pi^*]$ is well defined.
\end{proof}

\begin{rmk}\label{InverseImageVsPullbackBundles}
If $\Ll$ is a generalized line bundle on $Y$ and $D$ is a generalized divisor with fractional ideal $\Ii$ isomorphic to $\Ll$, then the inverse image divisor $\pi^*(D)$ has fractional ideal isomorphic to the pullback sheaf $\pi^*(\Ii)$ by Remark \ref{InverseImageVsPullback}. Since $\pi^*(\Ii)$ is isomorphic to $\pi^*(\Ll)$ as abstract $\Oo_X$-modules, we conclude that the inverse image $[\pi^*](\Ll)$ of the generalized line bundle $\Ll$ is equal to the pullback sheaf $\pi^*(\Ll)$.
\end{rmk}

We prove now a result on the composition of the direct image with the inverse image of generalized divisors.

\begin{prop}{\normalfont (Composition of the direct image with the inverse image)}\label{DirectAndInverseImage}
Let $D \in \GDiv(Y)$ be a generalized divisor on $Y$. Then, \[ \pi_*(\pi^*(D)) = n \cdot D. \]
\end{prop}
\begin{proof}
Since both of the terms are linear with respect to the sum of Cartier divisor, we can suppose that $D$ is effective with ideal sheaf $\Ii \subseteq \Oo_Y$. First note that, from the exact sequence:
\[ 0 \rightarrow \pi^* \Ii \rightarrow \pi^*\Oo_Y \rightarrow \pi^*\Oo_D \rightarrow 0 \]
together with Remark \ref{InverseImageVsPullback}, we obtain $\Oo_X/(\pi^{-1}\Ii \cdot \Oo_X) = \Oo_X/\pi^*\Ii = \pi^*(\Oo_Y/\Ii)$. To prove the thesis, we show that the equality \[
\Fitt_0\big(\pi_* (\Oo_X/\pi^*\Ii)\big)= \Ii^n
\] holds locally around any point $y \in Y$. Let $V \subseteq Y$ be an open neighborhood of $y$ such that  $(\pi_* \Oo_X)_{|V} \simeq (\Oo_{Y|V})^{\oplus n}$ and $\Ii_{|V}$ is generated by sections $s_1, \dots, s_r$ of $\Gamma(V, \Ii)$. Then, consider the following presentation:
\[ \Oo_{Y|V}^{\oplus r}  \xrightarrow{(\cdot s_1, \dots, \cdot s_r)} \Oo_{Y|V} \rightarrow (\Oo_Y/\Ii)_{|V} \rightarrow 0. \]
Pulling back with $\pi^*$, we obtain the following exact sequence on $U = \pi^{-1}(V)$:
\[ \Oo_{X|U}^{\oplus r}  \xrightarrow{(\cdot s_1, \dots, \cdot s_r)} \Oo_{X|U} \rightarrow (\Oo_X/\pi^*\Ii)_{|U} \rightarrow 0. \]
In order to compute $\Fitt_0\big(\pi_* (\Oo_X/\pi^*\Ii)\big)_{|V}$, we consider then the pushforward sequence:
\[ \big(\pi_*\Oo_X^{\oplus r}\big)_{|V}  \xrightarrow{(\cdot s_1, \dots, \cdot s_r)} \big(\pi_*\Oo_X\big)_{|V} \rightarrow \pi_*(\Oo_X/\pi^*\Ii)_{|V} \rightarrow 0. \]
Since $\pi_*(\Oo_X/\pi^*\Ii)_{|V} \simeq \big(\Oo_Y^{\oplus n}\big)_{|V}$, the map on the left is represented by the following $n \times nr$ matrix with entries in $\Gamma(V, \Oo_Y)$:
\[ M= 
\left[
\begin{array}{c|c|c}
\begin{matrix}
s_1 &  &   \\
 & \ddots &  \\
&  & s_1
\end{matrix} & \dots & \begin{matrix}
s_n &  &   \\
& \ddots &  \\
&  & s_n
\end{matrix}
\end{array}
\right].
\]	
Now, $\Fitt_0\big(\pi_* (\Oo_X/\pi^*\Ii)\big)_{|V}$ is generated by the $n \times n$ minors of $M$, i.e. all the possible products of $n$ generators of $\Ii$ on $V$, with ripetitions. This shows that $\Fitt_0\big(\pi_* (\Oo_X/\pi^*\Ii)\big)_{|V} = \Ii_{|V}^n$.
\end{proof}

\subsection{Relation with the degree}\label{DegreeOfDirectImage}

In this subsection we show that $\pi_*$ preserves the degree of divisors under the condition that $Y$ is smooth over $k$. In general, however, the direct image of a generalized divisor $D$ may not have the same degree of $D$, as the following example shows. Since we are dealing with degrees, we \textbf{assume that $X$ and $Y$ are projective curves over a base field} $k$.
\vspace{1em}


\begin{ex}\label{EsempioTacnode}
Fix $k$ an algebraically closed field. Let $A=k[x,y]/(y^2-x^4)$ be the affine coordinate ring of a curve $X=\Spec A$ with a tacnode at the point $P$ corresponding to the maximal ideal $\p=(x,y)$.

\noindent The involution $\sigma$ on $A$ defined by $x \mapsto -x, y \mapsto y$ induces a involution $\sigma_X$ on the curve $X$. The geometric quotient $Y=X/\sigma_X$ is an affine curve with coordinate ring equal to the ring of invariants $A^\sigma = k[x^2,y]/(y^2-x^4)$, that is isomorphic to $B=k[s,t]/(t^2-s^2)$ putting $s \mapsto x^2 $ and $t \mapsto y$. The quotient curve $Y$ has a simple node at the point $Q$ corresponding to the maximal ideal $\q=(s,t)$.

\noindent The inclusion map $A^\sigma \subset A$ gives to $A$ the structure of free $B$-module, with basis $\{1,x\}$; so, the corresponding morphism of curve $\pi: X \rightarrow Y = X/\sigma_X$ is a  finite, locally free map of degree $2$ sending $P$ to $Q$.

\noindent Let $D$ be the generalized divisor on $X$ defined by the ideal $I=(x^2,y) \subset A$; note that $D$ is supported only on the tacnode $P$. Since we want to compare $\deg(D)= \deg_P(D)$ with $\deg(\pi_*(D)) = \deg_Q(\pi_*(D))$, we can restrict to work locally around $P$ and $Q$. Let $B_\q$ be the local ring of $Y$ at $Q$. The induced map $B_\q \rightarrow A_\p$ makes $ A_\p$ a free $B_\q$-module of rank 2.

\noindent Let $E=A_\p/I_\p \simeq k[x]/(x^2)$ be the local ring of $D$ at the $P$. We have: \[
\deg_P(D) = \ell(E) = 2.
\]

\noindent Observe that $E$ has the following free presentation as $A_\p$-module:
\[
A_\p^{\oplus 2} \xrightarrow{(\cdot x^2, \cdot y)} A_\p \rightarrow E \rightarrow 0.
\]

\noindent Since $A_\p$ is a free $B_\q$-module of rank 2, $E$ has also a presentation as $B_\q$-module:
\[
B_\q^{\oplus 4} \xrightarrow{\varphi} B_\q^{\oplus 2} \rightarrow E \rightarrow 0
\]
where \[
\varphi = \left[\begin{array}{cc|cc}
s & 0 & t & 0 \\
0 & s & 0 & t
\end{array}\right].
\]
Then, the 0-th Fitting ideal of $E$ as $B_\q$-module is $F_0(E) = (s^2, t^2, st) \subset B_\q$. We have: \[
\deg_Q(\pi_*(D)) = \ell(B_\q/F_0(E)) = 3.
\]

\vspace{1em}

\noindent Note that there are divisors of degree 2 on $X$, supported at $P$, whose direct image has degree $2$ on $Y$. For example, take $D'=(x)$.

\end{ex}

\begin{rmk}
The previous example shows, in particular, that Proposition 2.33 in \cite{Vas} is false.
\end{rmk}

We now prove that the degree is preserved if the direct image is Cartier. We show first a proposition that computes the degree at any point where the direct image is locally principal.

\begin{prop}{\normalfont (Degree of the direct image of a generalized divisor)}\label{FiberDegree} 
Let $D \in \GDiv(X)$ be a  generalized divisor on $X$ and let $y \in Y$ be a point in codimension 1 of the support of  $\pi_*(D)$. Suppose that $\pi_*(D)$ is locally principal at $y$. Then,  \[
\deg_y(\pi_*(D)) = \sum_{\pi(x)=y} \deg_x(D).
\]
\end{prop}
\begin{proof} 
First, suppose that $D$ is effective. Let $V=\Spec(B)$ be an affine open neighborhood of $y$ with affine pre-image $U=\pi^{-1}(V)=\Spec(A)$, and let $I \subset A$ denote the ideal of $D$ restricted to $U$. The coordinate ring of $D$ on $U$ is the Artin ring  $A/I$ whose spectrum is equal to $\Spec(A) \cap \Supp(D) = \{ \p_1, \dots,  \p_s \}$; hence we have: \[ A/I = \prod_{i=1}^s (A/I)_{\p_i}. \] Let $\q \subset B$ denote the maximal ideal corresponding to $y$ in $\Spec(B)$ and let $B_\q$ be the associated local ring of dimension 1. Then, the localization of $A/I$ at $\q$, denoted $E$, is the coordinate ring of $D$ restricted to the fiber of $y$: \[
E = (A/I)_\q = \prod_{\pi^{-1}(\p_i)=\q} (A/I)_{\p_i}.
\]	

\vspace{1em}

\noindent Since $\pi_*(D)$ is effective, the degree of $\pi_*(D)$ at $y$ is: \begin{align*}
\deg_y(\pi_*(D)) &= \ell\big( \Oo_{Y,y} / F_0(\pi_* \Oo_D)_y \big)[\kappa(y):k] = \\
&= \ell\big( B_\q / F_0 (E)  \big)[\kappa(y):k].
\end{align*}
By hypothesis $F_0(E)$ is an invertible module, so by \cite[Proposition 2.32]{Vas}, we have $\ell\big( B_\q / F_0 (E) \big) = \ell(E)$.
 On the other hand, thanks to \cite[\href{https://stacks.math.columbia.edu/tag/02M0}{Tag 02M0}]{stacks-project}, 
note that: 
\begin{align*}
\ell(E) &= \ell \left( \prod_{\pi^{-1}(\p_i)=\q} (A/I)_{\p_i} \right) = \\
&= \sum_{\pi^{-1}(\p_i)=\q}\ell( A_{\p_i}/I_{\p_i} )][k(\p_i):k(\q)] = \\
&= \sum_{\pi(x)=y} \ell \big( \Oo_{X,x}/\Ii_x \big)[\kappa(x):\kappa(y)] = \\
&= \sum_{\pi(x)=y} \deg_x(D)[\kappa(x):\kappa(y)].
\end{align*}

\vspace{1em}
\noindent 
Putting everything together we have: 
\begin{align*}
\deg_y(\pi_*(D)) &= \ell(E) [\kappa(y):k] = \\
&= \left(\sum_{\pi(x)=y} \deg_x(D)[\kappa(x):\kappa(y)]\right)[\kappa(y):k]  = \\
&= \sum_{\pi(x)=y} \deg_x(D)[\kappa(x):k] = \\
&= \sum_{\pi(x)=y} \deg_x(D)
\end{align*}
for $D$ effective generalized divisor on $X$. 

\vspace{1em}

\noindent Let now $D=E-F$ be a generalized divisor, written as a difference of effective generalized divisors with $F$ Cartier by Lemma \ref{DifferenceOfEffective}. Since $\pi_*(F)$ is Cartier, using the result for effective divisors estabilished before, we have: \begin{align*}
\deg_y(\pi_*(D)) &= \deg_y(\pi_*(E)-\pi_*(F)) = \\
&= \deg_y(\pi_*(E)) - \deg_y(\pi_*(F)) = \\
&= \sum_{\pi(x)=y} \left[ \deg_x(E)-\deg_x(F) \right] = \\
&= \sum_{\pi(x)=y} \deg_x(D).
\end{align*}
\end{proof}

\begin{cor}\label{DegreeCor1}
Let $D \in \GDiv(X)$ be a generalized divisor on $X$ such that $\pi_*(D)$ is Cartier. Then, \[
\deg_Y(\pi_*(D)) = \deg_X(D).
\]
\end{cor}
\begin{proof}
Applying Definition \ref{DegreeGenDivDef} and Proposition \ref{FiberDegree}, we get:
\begin{align*}
\deg_Y(\pi_*(D)) &= \sum_{y \textrm{ cod 1}} \deg_y(\pi_*(D)) = \\
&= \sum_{y \textrm{ cod 1}} \sum_{\pi(x)=y} \deg_x(D).
\end{align*}
By the properties of the Fitting image, $\pi_*(D)$ is supported on the set-theoretic image of the support of $D$; hence, the $x$ appearing in the last sum are all the $x$ for which $\deg_x(D)$ is not zero. Then, the previous sum gives:
\begin{align*}
\deg_Y(\pi_*(D)) &= \sum_{x \textrm{ cod 1}} \deg_x(D) = \deg_X(D).
\end{align*}
\end{proof}

\begin{cor}\label{DegreeCor}
Suppose that $Y$ is smooth. Then, for any generalized divisor $D$ on $X$,  \[
\deg_Y(\pi_*(D)) = \deg_X(D).
\]
\end{cor}
\begin{proof}
Let $D$ be a generalized divisor on $X$. The direct image $\pi_*(D)$ is a generalized divisor on $Y$ smooth, hence Cartier. Then, apply Corollary \ref{DegreeCor1}.
\end{proof}

Proposition \ref{FiberDegree} yields another useful corollary about the surjectivity of the direct image morphism.

\begin{cor}{\normalfont (Surjectivity of the direct image for effective divisors)}\label{SurjDirectImage}
Suppose that $Y$ is smooth and $k$ is algebraically closed. Then, the direct image for effective divisors: \[
\pi_*^+: \GDiv^+(X) \rightarrow \CDiv^+(Y)
\]
is surjective.
\end{cor}
\begin{proof}
Let $E \in \CDiv^+(Y)$ be an effective Cartier divisor on $Y$ and let $V=\Spec(R) \subseteq Y$ be an affine open subset of $Y$ such that $E$ is supported on $V$. Let $\Supp(E)=\{ y_1, \dots, y_r \}$ be the support of $E$ and for each $i=1, \dots, r$ let \[
d_i=\deg_{y_i}(E)/[\kappa(y_i):k]=\ell_{\Oo_{Y,y_i}}(\Oo_{E,y_i}).
\] 
Let $U=\Spec(S)=\pi^{-1}(V)$ be the preimage of $V$. For each $i$, pick one element $x_i \in \pi^{-1}(y_i)$ in the finite fiber of $y_i$  and let $\mathfrak{M}_i$ be the maximal ideal of $S$ corresponding to the point $x_i$. Then, the ideal $I=\mathfrak{M}_1^{d_1} \cdot \dots \cdot \mathfrak{M}_r^{d_r}$ in $S$ defines a divisor $D$ on $U$ (and hence of $X$) such that $\deg_{y_i}(E)=\deg_{x_i}(D)$. Looking at its direct image $\pi_*(D)$, it is an effective divisor on $Y$ with the same support of $E$ and the same degree at any point by Propostion \ref{FiberDegree}. Since $Y$ is smooth, we conclude that $\pi_*(D)=E$.
\end{proof}

\section{The direct and inverse image for families of generalized divisors}\label{SectionDirectImageFamilies}
Let $\pi: X \rightarrow Y$ be a \textbf{finite, flat map of degree $n$ between projective curves over a field $k$}. In the present section, we discuss  the definition of direct and inverse image for families of effective generalized divisors. Under suitable conditions, recalling Definition \ref{HilbertDef}, we aim to define a pair of geometric morphisms: 
\begin{align*}
\pi_*&: \Hilb_X \rightarrow \Hilb_Y \\
\pi^*&: \Hilb_Y \rightarrow \Hilb_X
\end{align*}
that, on $k$-valued points, coincide with the direct and inverse image between $\GDiv^+(X)$ and $\GDiv^+(Y)$.

\vspace{1em}

\begin{rmk}
Giving a definition of the direct image for families of effective generalized divisors is not possible in general when the curve $Y$ is not smooth over $k$. Consider, for example, the setting of Example \ref{EsempioTacnode}; since $X$ and $Y$ are reduced curves with planar singularities, their Hilbert schemes of generalized divisors of given degree are connected (see \cite{AIK}, \cite{BGS}). The effective divisors $D_1$ and $D_2$ on $X$ defined by $(x^2, y)$ and $(x)$  on $X$  have both degree $2$, but their direct images on the quotient node $Y$ have degree respectively equal to $3$ and $2$. Then, their $k$-points $D_1, D_2$ in the connected component $\Hilb_X^2$ of $\Hilb_X$ are sent to different connected components of $\Hilb_Y$. This shows that \textit{the direct image of divisors in general is not defined as a geometric map}.
\end{rmk}

\vspace{1em}

In the rest of the section, consider $\pi: X \rightarrow Y$ a finite, flat map of degree $n$ between projective curves over $k$, and \textbf{suppose that $Y$ is smooth over $k$}. Recall that, in such case, $\Hilb_Y=\lHilb_Y$.

\begin{deflemma}{(Direct image for families of effective generalized divisors)}\label{DirectImageForHilbDef}
	Let $T$ be any $k$-scheme. The \textit{direct image map for the Hilbert scheme of effective generalized divisors} is defined on the $T$-valued points as:
	\begin{align*}
	\pi_*(T): \hspace{2em}	\Hilb_X(T)  &\longrightarrow \lHilb_Y(T) \\
	D \subseteq X \times_k T &\longmapsto \Zz\left( \Fitt_0(\pi_{T,*}(\Oo_D)) \right)
	\end{align*}
where $\pi_T: X \times_k T \rightarrow Y \times_k T$ is the morphism induced by base change of $\pi$.
\end{deflemma}
\begin{proof}
Let $D \subseteq X \times_k T$ be a $T$-flat family of effective divisors of $X$, defined by an ideal sheaf $\Ii\subseteq  \Oo_{X \times_k T}$ such that $\Oo_D=\Oo_{X \times_k T}/\Ii$ is flat over $S$. From the exact sequence \[
0 \rightarrow \Ii \rightarrow \Oo_{X \times_k T} \rightarrow \Oo_D \rightarrow 0
\]
we deduce that also $\Ii$ is flat over $T$. Since $\pi$ is finite and flat, $\pi_{T,*} (\Ii)$ is also flat over $T$, fiberwise locally free since $Y$ is smooth and hence locally free on $Y \times_k T$ by \cite[Lemma 2.1.7]{Huy}. Moreover, it fits the exact sequence \[
0 \rightarrow \pi_{T,*}(\Ii) \xrightarrow{\varphi} \pi_{T,*}(\Oo_{X \times_k T}) \rightarrow \pi_{T,*}(\Oo_D) \rightarrow 0.
\]
Then, by Definition \ref{FittingDef}, the $0$-th Fitting ideal of $\pi_{T,*}(\Oo_D)$ is the image of the canonical injection \[
\det\left( \pi_{T,*} (\Ii) \right) \otimes \det\left( \pi_{T,*}\Oo_{X\times_k T} \right)^{-1} \xhookrightarrow{\det(\varphi)} \Oo_{Y \times_k T}
\]
and this is locally free over $Y \times_k T$, hence flat over $T$. Then, it defines a $T$-flat family of effective divisors of $Y$.
\end{proof}

\begin{rmk}
	For any $T$-family of effective generalized divisors $D \subseteq X \times_k T$  and for any point $t \in T$, the fiber $\pi_*(T)(D)_t$ is equal to the direct image $\pi_*(D_t)$ defined for the effective divisor $D_t$ on $X$. Moreover, since $Y$ is smooth, by Corollary \ref{DegreeCor} we have: \[
	\deg_Y(\pi_*(D_t)) = \deg_X(D_t).
	\]
	Then, for any $d \geq 0$, $\pi_*$ restricts  to a map: \[
	\pi_*^d: \Hilb_X^d \longrightarrow \lHilb_Y^d.
	\]
\end{rmk}

\begin{deflemma}{(Inverse image for families of effective generalized divisors)}\label{InverseImageForHilb}
	Let $T$ be any $k$-scheme. The \textit{inverse image map for the Hilbert scheme of effective Cartier divisors} is defined on the $T$-valued points as:
\begin{align*}
\pi^*(T): \hspace{2em}	\lHilb_Y(T)  &\longrightarrow \lHilb_X(T) \\
D \subseteq Y \times_k T &\longmapsto \Zz\left( \pi_T^{-1}(\Ii_D)\cdot \Oo_{X \times_k T} \right)
\end{align*}
where $\pi_T: X \times_k T \rightarrow Y \times_k T$ is the morphism induced by base change of $\pi$ and $\Ii_D \subseteq  \Oo_{Y \times_k T}$ is the ideal sheaf of $D$.
\end{deflemma}
\begin{proof}
Since $\Ii$ is locally principal, $\pi_T^{-1}(\Ii)\cdot \Oo_{X \times_k T} \subseteq \Oo_{X \times_k T}$ is locally principal. The restriction of $\pi^*(T)(D)$ to the fiber over any $t \in T$ has ideal sheaf $\pi^{-1}(\Ii_t)\cdot \Oo_{X \times_k t}$, which is equal to the defining ideal of $\pi^*(D_t)$ by Definition \ref{InverseImageDef}. Hence, $\pi^*(T)(D)$ is a locally principal subscheme of  $X \times_k T$, such that all fibers over $T$ are effective Cartier divisors by Lemma \ref{InverseLinear}. Then, $\pi^*(T)(D)$ is $T$-flat  by \cite[\href{https://stacks.math.columbia.edu/tag/062Y}{Tag 062Y}]{stacks-project}, hence it is a $T$-family of Cartier divisors.
\end{proof}

\begin{rmk}
For any integer $d \geq 0$, $\pi^*$ restricts to a map: \[
\pi^*_d: \lHilb_Y^d \longrightarrow \lHilb_X^{nd}.
\]
\end{rmk}

We study now some properties of the direct and inverse image for families of effective generalized divisors. With a slight abuse of notation, we will write $\pi_*$ and $\pi^*$ instead of $\pi_*(T)$ and $\pi^*(T)$, when it is clear that we are working on $T$-points.

\begin{prop} {\normalfont (Properties of the direct and inverse image for families of effective generalized divisors)}
	Let $T$ be any $k$-scheme. Then, the following fact holds.
	\begin{enumerate}
		\item Let $D, E$ be $T$-families of effective divisors over $X$ such that $E$ is a family of Cartier divisors. Then, $\pi_*(D+E) = \pi_*(D)+\pi_*(E)$.
		\item Let $F, G$ be $T$-families of effective divisors over $Y$. Then, $\pi^*(F+G) = \pi^*(F)+\pi^*(G)$.
		\item If $F$ is a $T$-family of effective divisors over $Y$, then $\pi_*(\pi^*(F)) = nF$.
	\end{enumerate}
\end{prop}
\begin{proof}
The proof of parts (1) and (2) follows the proofs of the second part of Lemma \ref{linear} and \ref{InverseLinear} respectively. The proof of part (3) follows the proof of Proposition \ref{DirectAndInverseImage}.
\end{proof}

\section{The Norm  and the inverse image for families of torsion-free rank-1 sheaves}\label{SectionNormOnJbar}
In the present section, we provide the definition of the Norm map for families torsion-free sheaves of rank 1 and the related inverse image map. Since we are dealing with  families of sheaves, we  \textbf{assume that $X$ and $Y$ are projective curves over a field $k$}. 

\vspace{1em}

In order to be compatible with the direct image for generalized line bundles, our definition of the Norm map on $\Jbar(X)$ will be inspired by the sheaf-theoretic formula of Proposition \ref{DetFormula}. By Proposition \ref{GeneralizedDetFormula} and its corollaries, the generalization of such formula to generalized divisors and torsion-free sheaves involves taking the double $\omega$-dual of the exterior power of the pushforward of generalized line bundles. In general, this operation does not behave well in families if $Y$ is not smooth. 

Then, in accordance with the previous section, we \textbf{suppose that $Y$ is smooth over $k$}; in such case, the moduli space $\Jbar(Y)$ is actually equal to the Jacobian $J(Y)$. Then, the aim of this section is to provide a pair of geometric morphism: \begin{align*}
\Nm_\pi: \Jbar(X) &\rightarrow J(Y) \\
\pi^{-1}: J(Y) &\rightarrow J(X) \subseteq \Jbar(X)
\end{align*}
that, on $k$-valued points, coincide with the direct and inverse image maps between $\GPic(X)$ and $\GPic(Y)=\Pic(Y)$. 

Recall also (Definition \ref{TwistedAbel}) that, for any line bundle $M$ on $X$, the compactified Jacobian of $X$ is related to the Hilbert scheme of effective generalized divisors via the twisted Abel map $\A_M$. We will show that the direct image map and the Norm map are compatible as well as the inverse image maps, meaning that for any $M \in J(X)$ and $N \in J(Y)$ there are commutative diagrams of $k$-schemes:
\[ \begin{tikzcd}
\Hilb_X \arrow{rrr}{\pi_*} \arrow{d}{\A_M} &[-25pt]  &[-25pt]  & \lHilb_Y \arrow{d}{\A_{\Nm_\pi(M)}} \\
\Jgen(X) & \subseteq & \Jbar(X) \arrow{r}{\Nm_\pi} & J(Y) 
\end{tikzcd}
\hspace{4em}
\begin{tikzcd}
\lHilb_Y \arrow{r}{\pi^*} \arrow{d}{\A_N}  & \lHilb_X
\arrow{d}{\A_{\pi^*(N)}} \\
J(Y) \arrow{r}{\pi^*} & J(X).
\end{tikzcd}
\]

\vspace{1em}

We give first the definition for the Norm map on $\Jbar(X)$.

\begin{deflemma}{(Norm map for torsion-free rank-1 sheaves)}\label{NormOnMPureDef}
	Let $T$ be any $k$-scheme. The \textit{Norm map between compactified Jacobians} associated to $\pi$ is defined on the $T$-valued points as:
	\begin{align*}
	\Nm_\pi(T): \hspace{2em} \Jbar(X)(T) &\longrightarrow J(Y)(T) \\
	\Ll &\longmapsto \det\left( \pi_{T,*} (\Ll) \right) \otimes \det\left( \pi_{T,*}\Oo_{X\times_k T} \right)^{-1}.
	\end{align*}
\end{deflemma}
\begin{proof}
 Let $\Ll$ be a $T$-family of torsion-free sheaves of rank 1 on $X$, i.e. a $T$-flat coherent sheaf on $X \times_k T$, whose  fibers over $T$ are torsion-free sheaves of rank 1. The push-forward $\pi_{T,*} (\Ll)$ is a $T$-flat coherent sheaf on $Y \times_k T$ such that, for any $t \in T$, the fiber $(\pi_{T,*} \Ll)_t$ equals $\pi_*(\Ll_t)$ on $Y \simeq Y \times_k {t}$, . Since $Y$ is smooth,  $\pi_*(\Ll_t)$ is a locally free sheaf of rank $n$ for any $t \in T$. Then, by \cite[Lemma 2.1.7]{Huy}, $\pi_{T,*} (\Ll)$ is a locally sheaf of rank $n$ on $Y \times_k T $. Its determinant bundle is a line bundle on $Y \times_k T$, hence flat over $T$.
\end{proof}

The definition of the inverse image is standard but we give it for the sake of completeness.
\begin{deflemma}{(Inverse image map for line bundles)}\label{InverseImageOnCompJacDef}
Let $T$ be any $k$-scheme. The \textit{inverse image map between Jacobians} associated to $\pi$ is defined on the $T$-valued points as:
	\begin{align*}
\pi^*(T): \hspace{2em} J(Y)(T) &\longrightarrow J(X) \subseteq \Jbar(X)(T) \\
\Nn &\longmapsto \pi_T^*(\Nn).
\end{align*}
\end{deflemma}
\begin{proof}
By hypothesis, $\Nn$ is a $T$-flat coherent sheaf on $Y \times_k T$ that is a line bundle on any fiber over $T$; then, by \cite[Lemma 2.1.7]{Huy}, it is a line bundle on $Y \times_k T$. We conclude that $\pi^*\Nn$ is a line bundle on $X \times_k T$ and hence a $T$-flat family of line bundles over $T$.
\end{proof}

\begin{rmk} When there is no ambiguity, we will write $\Nm_\pi$ and $\pi^*$ instead of $\Nm_\pi(T)$ and $\pi^*(T)$. The Norm and the inverse image for generalized line bundles define morphisms of algebraic stacks \begin{align*}
\Nm_\pi: \hspace{2em} \Jbar(X) &\longrightarrow J(Y) \\
\pi^*: \hspace{2em} J(Y) &\longrightarrow J(X).
	\end{align*}
For any integer $d$, they restricts to: \begin{align*}
\Nm_\pi^d: \hspace{2em} \Jbar(X,d) &\longrightarrow J^d(Y) \\
\pi^*_d: \hspace{2em} J^d(Y) &\longrightarrow J^{nd}(X).
\end{align*}
Moreover, the Norm for torsion-free rank-1 sheaves, restricted to the locus of line bundles $J(X) \subseteq \Jbar(X)$, coincides with the classical Norm map from $J(X)$ to $J(Y)$ of Definition \ref{NormOnJacDef}.
\end{rmk}

We study now some properties of the Norm map on $\Jbar(X)$ and the inverse image map. First, we need a technical lemma.

\begin{lemma}\label{AssociatveFormulaForDet}
	Let $T$ be a fixed $k$-scheme. For any $T$-flat family  $\Ll$ of torsion-free rank-1 sheaves on $X \times_k T$ and any line bundle $\Mm$ on $X \times_k T$, there is an isomorphism: \begin{align*}
	\det\left( \pi_{T,*} (\Ll \otimes \Mm) \right) \otimes \det\left( \pi_{T,*}\Oo_{X\times_k T} \right) \simeq \\ \simeq \det\left( \pi_{T,*} (\Ll ) \right) \otimes  \det\left( \pi_{T,*} (\Mm) \right).
	\end{align*}
\end{lemma}
\begin{proof}
	The proof is similar to the second part of the proof of Proposition \ref{GeneralizedDetFormula}.
	
	Since $\pi_T: X \times_k T \rightarrow Y \times_k T$ is finite and flat, $\pi_{T,*}(\Ll)$ is a $T$-flat coherent sheaf on $Y \times_k T$, that is locally free on any fiber over $T$ since $Y$ is smooth; then, $\pi_{T,*}(\Ll)$ is locally free of rank $n$ by \cite[Lemma 2.1.7]{Huy}. The same holds for $\pi_{T,*}(\Ll \otimes \Mm)$. On the other hand, by \cite[Proposition 6.1.12]{EgaII}, $\pi_{T,*} \Mm$ is a locally free $\pi_{T,*}\Oo_X$-module of rank $1$. In particular, there is open cover $\{V_i\}_I$ of $Y$ such that $\Mm$ is a trivial $\Oo_{X|U_i}$-module on each $U_i = \pi_T^{-1}(V_i)$, i.e. there are isomorphisms $\lambda_i: \Mm_{|U_i}  \xrightarrow{{}_\sim} \Oo_{X|U_i}$ for each $i \in I$. On the intersections $U_i \cap U_j$, the collection $\{\lambda_i \circ \lambda_j^{-1} \}$ of automorphisms of ${\Oo_{X \times_k T}}_{|U_i \cap U_j}$ is a cochain that measures the obstruction for the $\lambda_i$'s to glue to a global isomorphism.  We define now an isomorphism \[
	\alpha: \det\left( \pi_{T,*} (\Ll \otimes \Mm) \right) \otimes \det\left( \pi_{T,*}\Oo_{X\times_k T} \right) 	\longrightarrow \det\left( \pi_{T,*} (\Ll ) \right) \otimes  \det\left( \pi_{T,*} (\Mm) \right)
	\] by glueing a collection of isomorphisms $\alpha_i$ defined on each $V_i$. To do so, we define each $\alpha_i$ as the following composition of arrows:
	\vspace{1em}
	
	\noindent\begin{tikzcd}
		\left(\det\big( \pi_{T,*} (\Ll \otimes \Mm) \big) \otimes \det(\pi_{T,*}\Oo_X)\right)_{|V_i} \arrow{r}{\alpha_i} \arrow{d}[phantom, anchor=center,rotate=-90,yshift=-1ex]{=}  &  \left(\det\big( \pi_{T,*} \Ll \big) \otimes \det\big( \pi_{T,*} \Mm \big)\right)_{|V_i} \\
		\det\big( \pi_{T,*} (\Ll \otimes \Mm)_{|U_i} \big) \otimes \det(\pi_{T,*}\Oo_{X|U_i}) \arrow{d}[swap]{\det\big(\pi_{T,*}\lambda_i\big) \otimes \id }  & \det\big( \pi_{T,*} \Ll_{|U_i} \big) \otimes \det\big( \pi_{T,*} \Mm_{|U_i} \big) \arrow{u}[anchor=center,rotate=-90,yshift=1ex]{=}  \\
		\det\big( \pi_{T,*} \Ll_{|U_i} \big) \otimes \det(\pi_{T,*}\Oo_{X|U_i}) \arrow[r, phantom, "="] &
		\det\big( \pi_{T,*} \Ll_{|U_i} \big) \otimes \det\big(\pi_{T,*}\Oo_{X|U_i} \big)
		\arrow{u}[swap]{\id \otimes \det \big(\pi_{T,*} \lambda_i^{-1}\big)} 
	\end{tikzcd}
	\vspace{1em}
	
	\noindent i.e. $\alpha_i := \det\big(\pi_{T,*}\lambda_i\big) \otimes \det \big(\pi_{T,*} \lambda_i^{-1}\big)$. Since $\det\big(\pi_{T,*}\_ \big)$ is functorial, the obstruction $\alpha_i \circ \alpha_j^{-1}$ is trivial on any $V_i \cap V_j$, whence the $\alpha_i$'s glue together to a global isomorphism $\alpha$.
\end{proof}

\begin{prop}{\normalfont (Properties of the Norm and the inverse image map)}\label{PropertiesOfNormOnCompJac}
	Let $T$ be any $k$-scheme. Then, the following facts hold.
	\begin{enumerate}
		\item Let $\Ll, \Mm \in \Jbar(X)(T)$  such that $\Mm$ is a $T$-flat family of line bundles. Then, $\Nm_\pi(\Ll \otimes \Mm) \simeq \Nm_\pi(\Ll) \otimes \Nm_\pi(\Mm)$.
		\item Let $\Nn, \Nn' \in J(Y)(T)$. Then, $\pi^*(\Nn \otimes \Nn') \simeq \pi^*(\Nn) \otimes \pi^*(\Nn')$.
		\item Let $\Nn \in J(Y)(T)$. Then, $\pi^*(\Nn)$ is a $T$-flat family of  line bundles over $X$ and $\Nm_\pi(\pi^*(\Nn)) \simeq \Nn^{\otimes n}$.
	\end{enumerate}
\end{prop}
\begin{proof}
	To prove (1), note that $\Mm$ is a $T$-flat coherent sheaf on $X \times_k T$ that is a line bundle on any fiber over $T$; hence it is a line bundle on $X \times_k T$ by \cite[Lemma 2.1.7]{Huy}. Then, applying Definition \ref{NormOnMPureDef} and Lemma \ref{AssociatveFormulaForDet}, we have:
	\begin{align*}
	\Nm_\pi(\Ll \otimes \Mm) &= \det\left( \pi_{T,*} (\Ll \otimes \Mm) \right) \otimes \det\left( \pi_{T,*}\Oo_{X\times_k T} \right)^{-1} \simeq \\
	&\simeq  \det\left( \pi_{T,*} (\Ll ) \right) \otimes  \det\left( \pi_{T,*} (\Mm) \right) \otimes \det\left( \pi_{T,*}\Oo_{X\times_k T} \right)^{-2}  \simeq \\
	&\simeq \Nm_\pi(\Ll) \otimes \Nm_\pi(\Mm).
	\end{align*}
	
\noindent	Part (2) follows from the associative properties of the tensor product.
	
\noindent	To prove (3), compute by Definitions \ref{NormOnMPureDef} and \ref{InverseImageOnCompJacDef}:
	\begin{align*}
	\Nm_\pi(\pi^*(\Nn)) = \det\left( \pi_{T,*} (\pi_T^*(\Nn)) \right) \otimes \det\left( \pi_{T,*}\Oo_{X\times_k T} \right)^{-1}.
	\end{align*}
	By the projection formula \cite[\href{https://stacks.math.columbia.edu/tag/01E8}{Tag 01E8}]{stacks-project} and the standard properties of determinants, we have:
	\begin{align*}
\Nm_\pi(\pi^*(\Nn)) &= \det\left( \Nn \otimes \pi_{T,*}\Oo_{X\times_k T}  \right) \otimes \det\left( \pi_{T,*}\Oo_{X\times_k T} \right)^{-1} \simeq \\
&\simeq  \Nn^n  \otimes  \det\left(\pi_{T,*}\Oo_{X\times_k T} \right) \otimes \det\left( \pi_{T,*}\Oo_{X\times_k T} \right)^{-1} \simeq \\
&\simeq \Nn^n.
\end{align*}	
\end{proof}

We compare now the Norm and the inverse image maps between the compactified Jacobians with the direct and inverse image maps between the Hilbert schemes of divisors, respectively.

\begin{prop}{\normalfont (Comparison of the direct image and the Norm map via the Abel map)}
For any line bundle $M$ of degree $e$ on $X$ and for any $d \geq 0$, there is a commutative diagram of algebraic stacks over $k$:
\[
\begin{tikzcd}
\Hilb_X^d \arrow{rrr}{\pi_*^d} \arrow{d}{\A_M^d} &[-25pt]  &[-25pt]  & \lHilb_Y^d \arrow{d}{\A_{\Nm_\pi(M)}^d} \\
\Jgen(X,-d+e) & \subseteq & \Jbar(X,-d+e) \arrow{r}{\Nm_\pi^{-d+e}} & J^{-d+e}(Y) 
\end{tikzcd}
\]
\end{prop}
\begin{proof}
Let $T$ be any $k$-scheme, let $D$ be a $T$-flat family of effective divisors of degree $d$ on $X$ with ideal sheaf $\Ii$, and denote with $\Mm$ the pullback of $M$ to $X \times_k T$.  Following the bottom-left side of the square, combining Definitions \ref{TwistedAbel} and \ref{NormOnMPureDef} we get: \[
 \Nm_\pi^{-d+e}(\A_M^d(D)) =  \det\left( \pi_{T,*} (\Ii \otimes \Mm) \right) \otimes \det\left( \pi_{T,*}\Oo_{X\times_k T} \right)^{-1}.
\] Following the top-right side of the square, combining definitions \ref{DirectImageForHilbDef}, \ref{NormOnMPureDef} and \ref{TwistedAbel} we get:\[
 \A_{\Nm_\pi(M)}^d(\pi_*^d(D)) = \det\left( \pi_{T,*} (\Ii ) \right) \otimes  \det\left( \pi_{T,*} ( \Mm) \right) \otimes \det\left( \pi_{T,*}\Oo_{X\times_k T} \right)^{-2}.
\]
We are left to prove that: 
\begin{align*}
\det\left( \pi_{T,*} (\Ii \otimes \Mm) \right) \otimes \det\left( \pi_{T,*}\Oo_{X\times_k T} \right) \simeq \\ \simeq \det\left( \pi_{T,*} (\Ii ) \right) \otimes  \det\left( \pi_{T,*} ( \Mm) \right).
\end{align*}
Now, $\Ii$ is a $T$-flat family of generalized line bundles by hypothesis and $\Mm$ is a line bundle on $X \times_k T$ since it is the pull-back of a line bundle on $X$. Then, the assertion is true by Lemma \ref{AssociatveFormulaForDet}.
\end{proof}

\begin{prop}{\normalfont (Comparison of the inverse image maps via the Abel map)}
	For any line bundle $N$ of degree $e$ on $Y$ and any $d \geq 0$, there is a commutative diagram of algebraic stacks over $k$:
	\[
	\begin{tikzcd}
	\lHilb_Y^d \arrow{r}{\pi^*_d} \arrow{d}{\A_N^d}  & \lHilb_X^d
	\arrow{d}{\A_{\pi^*(N)}^d} \\
	J(Y,-d+e) \arrow{r}{\pi^*_{-d+e}} & J(X,-d+e)
	\end{tikzcd}
	\]
\end{prop}
\begin{proof}
Let $T$ be any $k$-scheme, let $D$ be a $T$-flat family of divisors of degree $d$ on $Y$ with ideal sheaf $\Ii \subseteq Y \times_k T$, and denote with $\Nn$ the pullback of $N$ to $Y \times_k T$. By Definition \ref{TwistedAbel}, Proposition \ref{PropertiesOfNormOnCompJac}(2) and Remark \ref{InverseImageVsPullbackBundles}, we have: 
\begin{align*}
\pi^*_d(\A_N^d(D)) &= \pi_T^*(\Ii \otimes \Nn) \simeq \pi_T^*(\Ii) \otimes \pi_T^*(\Nn) \simeq \\
&\simeq \A_{\pi^*(N)}^d(\pi^*_{-d+e}(\Nn)).
\end{align*}
\end{proof}

\section{The fibers of the Norm map and the Prym stack}\label{SectionPrym}

In the present section we study the fibers of the Norm map defined in Section \ref{SectionNormOnJbar}. The purpose is to generalize the  Prym scheme associated to a finite, flat morphism between projective curves to the context of torsion-free sheaves of rank 1. Fix an algebraically closed field $k$ of characteristic 0. Let $X, Y$ be projective curves over $k$, such that $Y$ is smooth, and let $\pi: X \rightarrow Y$ be a finite, flat morphism of degree $n$. We start by recalling the following definition.
\begin{defi}
The \textit{Prym scheme} of $X$ associated to $\pi$  is the locus  $\Pr(X,Y) \subseteq J(X)$ of line bundles whose Norm with respect to $\pi$ is trivial. In other words: \[
\Pr(X,Y) = \{ \Ll \in J(X): \Nm_\pi(\Ll) \simeq \Oo_Y \}.
\]
\end{defi}

We can now extend the definition of the Prym stack to the context of torsion-free rank 1 sheaves using the Norm map defined in the previous section.

\begin{defi}
The \textit{Prym stack} of $X$  associated to $\pi$ is the locus  $\Prbar(X,Y) \subseteq \Jbar(X)$ of torsion-free rank-1 sheaves whose Norm with respect to $\pi$ is trivial. In other words: \[
\Prbar(X,Y) := \{ (\Ll, \lambda): \Ll \in \Jbar(X)\ \mathrm{ and }\ \lambda: \Nm_\pi(\Ll) \xrightarrow{{}_\sim} \Oo_Y \}.
\]
\end{defi}
\begin{rmk}
It follows from the definition that: \[
\Pr(X,Y) = \Prbar(X,Y) \cap J(X).
\]
\end{rmk}

Since being locally-free is an open condition,  $\Pr(X,Y)$ is open in $\Prbar(X,Y)$.

\vspace{1em}

The Prym stack is a fiber of the Norm map associated to $\pi$. In the present section, we study the fibers of the Norm map assuming that the curve $X$ is reduced with locally planar singularities. The study of such fibers involves the study of the fibers of the direct image map for effective divisors defined in Section \ref{SectionNormOnJbar}, and the study of the fibers of the Hilbert-Chow morphism. We start from the last.

\subsection{The fibers of the Hilbert-Chow morphism}
In this section, we refer to \cite{Ber} for notations and known results.
\begin{defi}
Let $X$ be a projective curve over $k$ and let $d$ be a positive integer. The $d$-\textit{symmetric power} $X^{(d)}$ of $X$ is  the quotient $X^d/\Sigma_d$ of the Cartesian product $X^n$ by the action of the symmetric group $\Sigma_d$ in $d$-letters permuting the factors. Note that, for $d=1$, $X^{(1)}=X$. For $d=0$, we put $X^{(0)}=\{0\}$.
The \textit{scheme of effective $0$-cycles} (or  of \textit{effective Weil divisors}) of $X$ is the algebraic scheme: \[
\WDiv^+(X) = \bigsqcup_{d \geq 0} X^{(d)} =  \bigsqcup_{d \geq 0} \WDiv^d(X).
\]
The \textit{Hilbert-Chow morphism}  of degree $d$ associated to $X$ is the map of schemes 
\begin{align*}
\rho_X^d: \Hilb_X^d &\rightarrow X^{(d)} \\
D &\mapsto \sum_{x \textrm{ cod 1}} \deg_x(D) \cdot [x]
\end{align*}
Denote with $\lrho_X^d$ the restriction of $\rho_X^d$ to the subscheme $\lHilb_X^d \subseteq \Hilb_X^d$ parametrizing Cartier divisors. The collection of Hilbert-Chow morphisms in non-negative degrees gives rise to the Hilbert-Chow morphism in any degree: \[
\rho_X:  \Hilb_X \rightarrow \WDiv^+(X).
\]
Finally, denote with $\lrho_X$ the restriction of $\rho_X$ to $\lHilb_X \subseteq \Hilb_X$.
\end{defi}

\begin{rmk}
If $X$ is a smooth projective curve, then $\rho_X$ is an isomorphism and coincides with $\lrho_X$.
\end{rmk}

We study now the fibers of the Hilbert-Chow morphism in the case when $X$ is  reduced  with locally planar singularities. 

\begin{prop}\label{fibersOfHC}
Suppose that $X$ is reduced with locally planar singularities and that $Y$ is smooth. Let $W \in \WDiv^d(X)$ be an effective Weil divisor of degree $d$ on $X$ and let $\rho_X^{-1}(W)$ be the corresponding fiber in $\Hilb_X$. The locus $\lrho_X^{-1}(W)$ of Cartier divisors in $\rho_X^{-1}(W)$ is an open and dense subset of $\rho_X^{-1}(W)$.
\end{prop}
\begin{proof}
	First, note that $\lrho_X^{-1}(W) = \rho_X^{-1}(W) \cap \lHilb_X$. Since $\lHilb_X \subseteq \Hilb_X$ is open (see for example \cite[Fact 2.4]{MRVFine}), $\lrho_X^{-1}(W)$ is open in $\rho_X^{-1}(W)$.
	
	To prove that it is dense, let $W= \sum_{i=1}^s n_i \cdot [x_i]$ with the $x_i$'s are $s$ distinct points. For each $i$, the fiber of $\rho_X$ over the cycle $n_i \cdot [x_i]$ is equal by definition to the punctual Hilbert scheme $\Hilb_{X,x_i}^{n_i}$  parametrizing $0$-dimensional subschemes of $X$ supported at $x_i$ having length $n_i$ over $k$. Since the $x_i$ are distinct, we have: \begin{align*}
	\rho_X^{-1}(W) &=\rho_X^{-1}\left( \sum_{i=1}^s n_i \cdot [x_i]\right) =\\
	&=\prod_{i=1}^s \rho_X^{-1}(n_i \cdot [x_i]) = \prod_{i=1}^s \Hilb_{X,x_i}^{n_i}.
	\end{align*}
	For each $i$, let  $\lHilb_{X,x_i}^{n_i}=\Hilb_{X,x_i}^{n_i} \cap \lHilb_X$. An effective divisor is Cartier if and only if it is Cartier at each point of its support, hence: \[
	\lrho_X^{-1}(W)=\prod_{i=1}^s \left(\Hilb_{X,x_i}^{n_i} \cap \lHilb_X\right) = \prod_{i=1}^s \lHilb_{X,x_i}^{n_i}.
	\]
	For each $i$, $\lHilb_{X,x_i}^{n_i}$ is the smooth locus of $\Hilb_{X,x_i}^{n_i}$ by \cite[Proposition  2.3]{BGS}, so it is dense; hence $\lrho_X^{-1}(W)$ is dense in $\rho_X^{-1}(W)$.
\end{proof}

\subsection{The fibers of the direct image between Hilbert schemes}

In the present subsection, we study the fibers of the direct image map $\pi_*$ defined between the Hilbert schemes of generalized divisors of $X$ and $Y$. 

\vspace{1em}

First, we introduce a similar notion for Weil divisors and we see how it relates to the direct image for generalized divisors.

\begin{defi}
The \textit{direct image for Weil divisors} associated to $\pi$ is the morphism of schemes: \[
\pi_*: \WDiv^+(X) \rightarrow \WDiv^+(Y)
\]
given on the level of points by \[
\sum_{i=1}^d n_i \cdot [x_i] \longmapsto \sum_{i=1}^d n_i \cdot [\pi(x_i)].
\]
\end{defi}

\begin{prop}\label{DirectImageVsChow}
Assume that $Y$ is smooth. There is a commutative square of schemes over $k$:
\[
\begin{tikzcd}
\Hilb_X \arrow{r}{\pi_*} \arrow{d}{\rho_X}  & \lHilb_Y
\arrow{d}{\rho_Y} \\
\WDiv(X) \arrow{r}{\pi_*} & \WDiv(Y).
\end{tikzcd}
\]
\end{prop}
\begin{proof}
Let $D$ be an effective generalized divisor on $X$. Following the bottom-left side of the square, we obtain: \[
\pi_*(\rho_X(D)) = \sum_{x \textrm{ cod 1}} \deg_x(D) \cdot [\pi(x)],
\]
while on the top-right side of the square we have \[
\rho_Y(\pi_*(D)) = \sum_{y \textrm{ cod 1}} \deg_y(\pi_*(D)) \cdot [y].
\]
Since $Y$ is smooth, $\pi_*(D)$ is Cartier locally at each point of $Y$. By Proposition \ref{FiberDegree}, for any point $y$ in codimension 1 of $Y$ \[
\deg_y(\pi_*(D)) = \sum_{\pi(x)=y} \deg_x(D).
\]
Then we compute: \begin{align*}
\rho_Y(\pi_*(D)) &= \sum_{y \textrm{ cod 1}} \deg_y(\pi_*(D)) \cdot [y] = \\
&= \sum_{y \textrm{ cod 1}} \left(\sum_{\pi(x)=y} \deg_x(D)\right) \cdot [y] = \\
&= \sum_{y \textrm{ cod 1}} \sum_{\pi(x)=y} \deg_x(D) \cdot [\pi_*(x)] = \\
&= \sum_{x \textrm{ cod 1}}  \deg_x(D) \cdot [y] = \pi_*(\rho_X(D)).
\end{align*}
\end{proof}

We are now ready to study the fibers of the direct image for generalized divisors.

\begin{prop}\label{fibersOfPiStar}
Assume that $Y$ is smooth and $X$ is reduced with locally planar singularities and . Let $E \in \Hilb_Y$ be an effective divisors of degree $d$ on $Y$ and let $\pi_*^{-1}(E)$ be the corresponding  fiber in $\Hilb_X$. Then, $\pi_*^{-1}(E) \cap \lHilb_X$ is an open and non-empty dense subset of $\pi_*^{-1}(E)$.
\end{prop}
\begin{proof}
By Proposition \ref{DirectImageVsChow}, there is a commutative diagram of schemes over $k$: \[
\begin{tikzcd}
\Hilb_X \arrow{r}{\pi_*} \arrow{d}{\rho_X}  & \lHilb_Y
\arrow{d}{\rho_Y} \\
\WDiv(X) \arrow{r}{\pi_*} & \WDiv(Y).
\end{tikzcd}
\]
First note that, since $Y$ is smooth, the map $\pi_*$ is surjective by Corollary \ref{SurjDirectImage}, hence $\pi_*^{-1}(E)$ is non-empty. Moreover, the Hilbert-Chow morphism $\rho_Y$ is an isomorphism. Then, \[
\pi_*^{-1}(E)=\rho_X^{-1}\pi_*^{-1}(\rho_Y(E)).
\]
Let $S=\{ y_1, \dots, y_r\}$ be the support of $E$ and let $d_i= \deg_{y_i} E$ for each $i$; then \[
\rho_Y(E) = \sum_{i=1}^r d_i \cdot [y_i].
\]
For each $i$, denote with $\{ x_i^1, \dots, x_i^{n_i} \}$ the discrete fiber of $\pi$ over $y_i$. Since points can occur with multeplicity in the geometric fibers, $n_i \leq n$ holds for each $i$. The fiber of $ \sum_{i=1}^r d_i \cdot [y_i]$ is then the discrete subset of $\WDiv(X)$ given by: \[
\pi_*^{-1}\left( \sum_{i=1}^r d_i \cdot [y_i]\right)= \left\{\sum_{i=1}^{r} \sum_{j=1}^{n_i} c_{ij} \cdot [x_i^j] \ \middle|\ c_{ij} \in \Z_{\geq 0}, \  \sum_{j=1}^{n_i} c_{ij} = d_i \textrm{ for any } i\right\}.
\]

Denote with $Z$ the set of all tuples $(c_{ij})$ of positive integers with the conditions that $\sum_{j=1}^{n_i} c_{ij} = d_i$ for any $i$.  Then, we can write the previous set as \[
\pi_*^{-1}\left( \sum_{i=1}^r d_i \cdot [y_i]\right) = \bigsqcup_{(c_{ij}) \in Z} \left\{ \sum_{i=1}^{r} \sum_{j=1}^{n_i} c_{ij} \cdot [x_i^j] \right\}.
\]
Then, applying $\rho_*^{-1}$, we otbain: \begin{align*}
\rho_X^{-1}\pi_*^{-1}(\rho_Y(E)) &= \rho_X^{-1}\pi_*^{-1}\left( \sum_{i=1}^r d_i \cdot [y_i]\right) =\\
&= \rho_X^{-1}\left( \bigsqcup_{(c_{ij}) \in Z} \left\{ \sum_{i=1}^{r} \sum_{j=1}^{n_i} c_{ij} \cdot [x_i^j] \right\} \right) = \\
&= \bigsqcup_{(c_{ij}) \in Z} \rho_X^{-1}\left(   \sum_{i=1}^{r} \sum_{j=1}^{n_i} c_{ij} \cdot [x_i^j]  \right),
\end{align*}
where the last disjoint union is a disjoint union of topological subspaces of $\Hilb_X$ in the same connected component. Then, $\pi_*^{-1}(E)$ is a non-emtpy disjoint uniont of a finite number of fibers of the Hilbert-Chow morphism. By Proposition \ref{fibersOfHC}, the intersection of any such fiber with the Cartier locus $\lHilb_X$ is open and dense in the same fiber.  Taking the disjoint union,  $\pi_*^{-1}(E) \cap \lHilb_X$ is a non-empty open and dense subset of $\pi_*^{-1}(E)$.
\end{proof}

\subsection{The fibers of the Norm map and $\Prbar(X,Y)$}
In this subsection, assuming that $Y$ is smooth and $X$ is reduced with locally planar singularities, we prove that the generalized Prym of $X$ with respect to $Y$ is non-empty, open and dense in the Prym stack of $X$ with respect to $Y$. 
The theorem is based upon the following two auxiliaries proposition. 

\begin{prop}\label{SurjNorm}
Assume that $Y$ is smooth and $X$ is reduced with locally planar singularities and let $\Nn \in J(Y)$ be any line bundle on $Y$. Then, there is a line bundle $\Ll$ on $X$ such that $\Nm_\pi(\Ll) \simeq \Nn$.
\end{prop}
\begin{proof}
Let $D$ be a Cartier divisor on $Y$ such that $\Nn \in [D]$ and let $D=E-F$ where $E,F$ are effective Cartier divisors by Lemma \ref{DifferenceOfEffective}. By Proposition \ref{fibersOfPiStar} there are effective Cartier divisors $H,K$ on $X$ such that $\pi_*(H)=E$ and $\pi_*(K)=F$. Then, by Lemma \ref{linear} $\pi_*(H-K)=D$. Set $H-K=C$ and let $\Ll$ be the defining ideal of $C$, seen as a line bundle.  By Lemma \ref{NmVsPiStar} we conclude that $\Nm_\pi(\Ll)\simeq \Nn$.
\end{proof}

\begin{prop}\label{fibersOfNm}
Assume that $Y$ is smooth  and $X$ is reduced with locally planar singularities and let $\Nn \in J(Y)$ be any line bundle on $Y$. Then, the fiber $\Nm_\pi^{-1}(\Nn)$ is non-empty and contains $\Nm_\pi^{-1}(\Nn) \cap J(Y)$ as an open and dense substack.
\end{prop}
\begin{proof}
First, note that $\Jbar(X)=\Jgen(X)$ since $X$ is reducible. The fiber $\Nm_\pi^{-1}(\Nn)$ is non-empty by Proposition \ref{SurjNorm} and the substack $\Nm_\pi^{-1}(\Nn) \cap J(Y)$ is open in $\Nm_\pi^{-1}(\Nn)$ since being locally free is an open condition. To prove that it is dense, recall that for any fixed line bundle $M \in J(X)$ there is a commutative diagram of stacks over $k$:
\[
\begin{tikzcd}
\Hilb_X \arrow{r}{\pi_*} \arrow{d}{\A_M}  & \lHilb_Y
\arrow{d}{\A_{\Nm_\pi(M)}} \\
\Jbar(X) \arrow{r}{\Nm_\pi} & J(Y)
\end{tikzcd}
\]
Moreover, by \cite[Proposition 2.5]{MRVFine} there is a cover of $\Jbar(X)$ by $k$-finite type open subsets $\{ U_\beta \}$ such that, for each $\beta$, there is a line bundle $M_\beta \in J(X)$ with the property that $\A_{M_\beta|V_\beta}: \A_{M_\beta}^{-1}(U_\beta)=V_\beta \rightarrow U_\beta$ is smooth and surjective.

Since density can be checked locally, fix $M=M_\beta$, $U=U_\beta$ and $V=V_\beta$. Then, we have: \[
\Nm_\pi^{-1}(\Nn) \cap U = \A_M\left( V \cap \pi_*^{-1}\left(\A_{\Nm_\pi(M)}^{-1}(\Oo_Y)\right) \right)
\] and \[
\Nm_\pi^{-1}(\Nn) \cap J(X) \cap U = \A_M\left( V \cap \pi_*^{-1}\left(\A_{\Nm_\pi(M)}^{-1}(\Oo_Y)\right) \cap \lHilb_X \right).
\]
Put $\textbf{K}=\A_{\Nm_\pi(M)}^{-1}(\Oo_Y)$. Then,  the topological space underlying $\pi_*^{-1}(\textbf{K})$ contains the topological union of fibers of the points in $\textbf{K}$, so that: \[
\pi_*^{-1}(\textbf{K}) \supseteq \bigcup_{E \in \textbf{K}} \pi_*^{-1}(E).
\]
By Proposition \ref{fibersOfPiStar}, $\pi_*^{-1}(E) \cap \lHilb_X$ is a non-empty open and dense subscheme of $\pi_*^{-1}(E)$ for any closed point $E \in \textbf{K}$. Hence, $\pi_*^{-1}(\textbf{K}) \cap \lHilb_X$ is non-empty, open and dense in $\pi_*^{-1}(\textbf{K})$. Intersecting with $V$ and composing with $\A_M$, we get the thesis.
\end{proof}

We finally come to Prym stack of $X$ with respect to $Y$. Recall that by definition: \begin{align*}
\Pr(X,Y) &=\Nm_\pi^{-1}(\Oo_Y) \cap J(X), \\
\Prbar(X,Y) &= \Nm_\pi^{-1}(\Oo_Y).
\end{align*}
\begin{cor}
$\Pr(X,Y)$ is non-empty, open and dense in $\Prbar(X,Y)$.
\end{cor}
\begin{proof}
Set $\Nn = \Oo_Y$ and apply Proposition \ref{fibersOfNm}.
\end{proof}

\bibliographystyle{alpha}
\bibliography{norm3}

\end{document}